\documentclass[10pt]{amsart}
\calclayout
\usepackage{amsmath,amsthm,amssymb,latexsym,ifthen,graphicx}
\usepackage{mathrsfs}
\usepackage[most]{tcolorbox}
\usepackage{color}
\usepackage[colorlinks,citecolor=black,urlcolor=black, linkcolor=black]{hyperref}

\usepackage[T1]{fontenc}
\usepackage[utf8]{inputenc}
\usepackage{lmodern}
\usepackage{microtype}
\usepackage{enumitem}
\usepackage{needspace}
\usepackage{xcolor}
\usepackage{tikz}
\usetikzlibrary{arrows.meta}
\usepackage{aliascnt}

\newtheorem{theorem}{Theorem}[section]
\newaliascnt{lemma}{theorem}
\newtheorem{lemma}[lemma]{Lemma}
\aliascntresetthe{lemma}
\newaliascnt{proposition}{theorem}

\aliascntresetthe{proposition}
\newaliascnt{corollary}{theorem}
\newtheorem{corollary}[corollary]{Corollary}
\aliascntresetthe{corollary}
\newaliascnt{claim}{theorem}

\aliascntresetthe{claim}
\newaliascnt{fact}{theorem}
\newtheorem{fact}[fact]{Fact}
\aliascntresetthe{fact}
\newtheorem*{historicalconjecture}{Conjecture}
\theoremstyle{definition}
\newaliascnt{definition}{theorem}
\newtheorem{definition}[definition]{Definition}
\aliascntresetthe{definition}
\newaliascnt{remark}{theorem}

\aliascntresetthe{remark}
\newaliascnt{question}{theorem}
\newtheorem{question}[question]{Question}
\aliascntresetthe{question}

\usepackage[nameinlink,noabbrev]{cleveref}
\crefname{theorem}{theorem}{theorems}
\Crefname{theorem}{Theorem}{Theorems}
\crefname{lemma}{lemma}{lemmas}
\Crefname{lemma}{Lemma}{Lemmas}
\crefname{proposition}{proposition}{propositions}
\Crefname{proposition}{Proposition}{Propositions}
\crefname{corollary}{corollary}{corollaries}
\Crefname{corollary}{Corollary}{Corollaries}
\crefname{claim}{claim}{claims}
\Crefname{claim}{Claim}{Claims}
\crefname{fact}{fact}{facts}
\Crefname{fact}{Fact}{Facts}
\crefname{definition}{definition}{definitions}
\Crefname{definition}{Definition}{Definitions}
\crefname{remark}{remark}{remarks}
\Crefname{remark}{Remark}{Remarks}
\crefname{question}{question}{questions}
\Crefname{question}{Question}{Questions}

\newcommand{\F}{\mathbb F}
\newcommand{\R}{\mathbb R}
\newcommand{\C}{\mathbb C}
\newcommand{\N}{\mathbb N}
\newcommand{\Q}{\mathbb Q}

\newcommand{\eps}{\varepsilon}
\newcommand{\dd}{\,d}
\newcommand{\restr}{\mathord{\upharpoonright}}
\newcommand{\Age}{\operatorname{Age}}
\newcommand{\Emb}{\operatorname{Emb}}
\newcommand{\Iso}{\operatorname{Iso}}

\newcommand{\ldis}{\operatorname{ldis}}
\newcommand{\distBM}{d_{\mathrm{BM}}}
\newcommand{\id}{\operatorname{Id}}

\title[Computable Presentations of Separable Lebesgue Spaces]
{Computable Presentations of Separable Lebesgue Spaces}
\author{Ruofei Xie}
\author{Xuanheng Zhao}
\address{School of Mathematics, Nanjing University,
Nanjing, Jiangsu 210093, People's Republic of China}
\email{xie.ruofei@nju.edu.cn}
\email{xuanheng21@gmail.com}

\subjclass[2020]{Primary 03D78; Secondary 03D45, 46B04, 46E30}
\keywords{Computable presentations, 
Lebesgue spaces,
approximate ultrahomogeneity}

\begin{document}

\begin{abstract}
For every right-c.e.\ real $p\in(1,2)$, $L^p[0,1]$ has a computable presentation.
The proof combines an effective stability modulus for almost isometric
embeddings with a Fra\"{\i}ss\'e-style construction in which the exponent varies.
In the same range,  if
$\ell^p$ has a computable presentation, then $p$ is computable.  Combining these with the known results gives the 
computable presentability classification for all separable Lebesgue spaces.
We also formulate an effective Fra\"{\i}ss\'e theorem for Banach spaces and
apply it to obtain a computable presentation of the Gurarij Banach space.
\end{abstract}

\maketitle

\section{Introduction}

The effective study of continuous mathematics begins with a
deceptively simple question: what should it mean to compute a real
number?  In his landmark paper \cite{Turing}, Turing gave
this question a precise answer by replacing a completed infinite
expansion with a finite mechanical procedure that produces its digits
one by one.  Although he used this definition to address the
Entscheidungsproblem, the conceptual move reached far beyond its
original application.  It supplied a lasting paradigm for computable analysis: an infinite
object becomes accessible through a \emph{uniform} process that generates
finite approximations.  The same shift of viewpoint---from an object
in the abstract to the information by which it is presented---also
lies at the heart of computable structure theory.  The present paper
follows this line of thought for $L^p$-spaces.

The distinction between an abstract mathematical object and a particular presentation of it is largely suppressed in classical mathematics.
Isomorphic groups, isometric metric spaces, and linearly isometric Banach spaces
are regarded as the same object, and a theorem is normally stated
about the abstract object rather than about a particular choice of
coordinates.  Computation cannot ignore this distinction.  An
algorithm never receives an abstract isomorphism type, and an infinite
object cannot be handed to it all at once.  What the algorithm receives
is a concrete \emph{presentation}: a system of names together with
procedures for evaluating the relevant operations, relations, or
distances.  
Computable structure theory studies mathematics
through such presentations.  It asks for the effective content of
classical results: which infinite structures admit algorithmic
descriptions, which invariants and constructions can be recovered from
those descriptions, whether classically equivalent descriptions can
be translated into one another effectively, and how much additional
computational information is required when they cannot.  In this way
the subject connects the computability theory of countable algebraic
structures with computable analysis and the study of separable
uncountable structures, see
\cite{DM}
and \cite{FGM}.

For a countable structure, a computable presentation gives an
isomorphic copy on a computable subset of $\N$ with computable basic
operations and relations. 
The passage to analysis creates an immediate obstacle.  A separable
Banach space is normally uncountable, so its vectors cannot be listed
by natural numbers.  Separability nevertheless supplies a countable
interface: if $(v_n)_{n\in\N}$ has dense linear span, then its rational
linear combinations form a countable skeleton whose completion is the
whole space.  Computability requires effective access to the norm on
this skeleton; arbitrary vectors and operators are then handled by
Cauchy names. 

This framework gives rise to three closely related problems.  The
\emph{presentability problem} asks which classical structures admit
at least one computable presentation.  The \emph{effective uniqueness
problem} asks whether all computable presentations of a fixed
isomorphism or isometry type are computably equivalent.  The
\emph{recognition problem} ranges over presentations and asks how
difficult it is to determine whether the presented object belongs to a
specified class or has a specified type.
Separable Lebesgue spaces provide a particularly revealing test case.  Their
classical linear-isometric structure is governed by the exponent and
the underlying measure algebra, but an arbitrary computable Banach
presentation need not display either of these objects.  It presents
only a dense sequence and the norm geometry generated by that
sequence.  The exponent must therefore be detected indirectly, if it
can be detected at all, through geometric identities, embeddings, and
finite-dimensional configurations.  Conversely, to build a
computable presentation one must make all norm computations effective
even when the exponent itself is not computable. 
 
 The history leading
to our work passes naturally through the three problems above.
We begin with effective uniqueness.  A computably (presentable) Banach space is \emph{computably categorical} if
any two computable Banach presentations of spaces linearly isometric
to it are connected by a surjective linear isometry that is computable
with respect to their induced Cauchy representations. Pour-El and Richards \cite{PER} proved that every separable Hilbert space is computably categorical.
The first examples beyond Hilbert spaces arose in finite dimensions.
Let $(\Omega,\Sigma,\mu)$ be a measure space and
$1\le r<\infty$.  The \emph{Lebesgue space}
$L^r(\Omega,\Sigma,\mu;\F)$ consists of equivalence classes modulo
almost-everywhere equality of measurable $f:\Omega\to\F$ for which
$\|f\|_r=\left(\int_\Omega |f|^r\dd\mu\right)^{1/r}<\infty$.
We abbreviate $L^r([0,1],\lambda;\F)$ to $L^r[0,1]$.  Also,
\[
 \ell_r^N=\left(\F^N,
 \|(a_j)_{j<N}\|_r=\left(\sum_{j<N}|a_j|^r\right)^{1/r}\right).
\]
Informally, a real number is \emph{computable} if a single algorithm
can approximate it to any prescribed accuracy.  Thus, when
$p$ is computable, the exponent can be used effectively in norm
calculations.  McNicholl \cite{McNichollCopies} established computable categoricity for $\ell_p^N$ where $p \ge 1$ is computable.

Clanin, McNicholl, and Stull subsequently obtained an
infinite-dimensional counterpart.  For computable $p$, the space
$L^p[0,1]$ has a standard computable
presentation given by the step functions with coefficients in
$\Q$ in the real case and in $\Q+i\Q$ in the complex case, and with
rational interval endpoints.

\begin{theorem}[Clanin--McNicholl--Stull \cite{CMS}]\label{thm:nonatomic-categoricity}
Let $\F$ be either $\R$ or $\C$, and let $p\ge 1$ be a computable
real.  If $\Omega$ is an atomless, separable measure space such that
$L^p(\Omega;\F)\ne\{0\}$, then every computable presentation of
$L^p(\Omega;\F)$ is
computably linearly isometric to the standard computable presentation
of $L^p([0,1];\F)$.  Consequently, $L^p([0,1];\F)$ is computably
categorical for every computable real $p\ge 1$.
\end{theorem}

 We next turn to the recognition problem: how difficult is it to
recognize a Lebesgue space from an index? Here a \emph{Lebesgue space} is a separable
Banach space linearly isometric to $L^p(\Omega)$ for some measure space
$\Omega$ and some exponent $p$.
A computable Banach-space presentation may be viewed as
a program that, given a rational linear combination of its
distinguished dense sequence and a requested precision, computes the
norm of that combination.  An \emph{index} is simply a natural-number
code for such a program.  Neither the presentation nor its index is
unique.  Given a class of Banach spaces, its \emph{index
set} is the collection of all codes whose programs present members of
that class.  The position of this set in the arithmetical or analytical
hierarchy measures how difficult it is to recognize the class from an
effective description.
  Naively, deciding whether an indexed
Banach space is a Lebesgue space appears to require a search through
the uncountably many possible surjective linear isometries and hence to have the crude
analytic upper bound $\Sigma^1_1$.

\begin{theorem}[Brown--McNicholl--Melnikov \cite{BMM}]\label{thm:lebesgue-index-sets}
The following statements hold.
\begin{enumerate}[label=(\roman*)]
\item The set of all indices of Lebesgue-space presentations, with the
exponent allowed to vary, belongs to $\Pi^0_3$.
\item For every fixed computable real $p \ge 1$, the set of all indices
of $L^p$-space presentations is $\Pi^0_2$-complete.
\item For every fixed computable real $p\ge 1$, the set of all indices
of presentations linearly isometric to $L^p[0,1]$ is
$\Pi^0_2$-complete.
\end{enumerate}
\end{theorem}
 
A one-sided recovery theorem for the uncomputable exponents 
is also the starting point for our work.

\begin{theorem}[\cite{BMM}]\label{thm:one-sided-exponent-recovery}
Let $B$ be a $L^p$-space ($1 \le p \le 2$) of dimension at least two as a linear space.  Uniformly from a presentation of $B$,
one can computably enumerate $\{r\in\Q:p<r\}$.
\end{theorem}

A real number $x$ is \emph{right-c.e.\ }if $\{r\in\Q:x<r\}$
is c.e., and it is \emph{left-c.e.\ }if
$\{r\in\Q:r<x\}$ is c.e.  It is \emph{computable}
if it is both left-c.e.\ and right-c.e. 
Consequently, if $1<p<2$ and an $L^p$-space of dimension at least
two is computably presentable, then $p$ must be right-c.e. The restriction on the dimension is necessary, since every
one-dimensional $L^p$-space is linearly isometric to the scalar field
and therefore carries no information about $p$.

\begin{theorem}[McNicholl \cite{McNichollExponent}]\label{thm:mcnicholl-exponent-recovery}
The following statements hold.
\begin{enumerate}[label=(\roman*)]
\item If $2\leq p<\infty$ and a computably presentable $L^p$-space has
dimension at least two, then $p$ is computable.
\item If $n\geq2$, $1 \le p<\infty$, and $\ell_n^p$ is computably
presentable, then $p$ is computable.
\end{enumerate}
\end{theorem}

McNicholl also proved that exponent recovery is at least nonuniform.

\begin{theorem}[\cite{McNichollExponent}]\label{thm:mcnicholl-embedding-nonuniformity}
For every nondegenerate open interval $I\subseteq(1,2)$, there
is no partial computable function $F$ such that, whenever $e$ indexes
a computable presentation of an $L^p$-space of dimension at least two
with $p\in I$, $F(e)$ is defined and indexes a computable Cauchy name
for $p$.
\end{theorem}

The remaining case is that of infinite-dimensional separable
$L^p$-spaces with $1<p<2$. The one-sided recovery
theorem makes right-c.e.-ness of $p$ necessary, while the
nonuniformity theorem suggests that computability may not be
necessary.

\begin{historicalconjecture}[McNicholl \cite{McNichollExponent}, 2020]
If $1<p<2$ is right-c.e., then $L^p[0,1]$ is computably presentable.
\end{historicalconjecture}

This conjecture was later recorded as Conjecture~6.33 in the
textbook by Franklin, Goldbring, and McNicholl \cite{FGM}.
It is also connected with the following question in effective
duality, posed by Downey and Melnikov in their textbook
\cite[Question~4.2.116]{DM}; a broader version was also posed by
Downey, Greenberg, and Qian \cite[Question~1]{DGQ}.

\begin{question}[Downey--Melnikov]\label{question:computable-dual}
Suppose $B$ is a computable Banach space and its dual $B^*$ is
separable.  Must $B^*$ have a computable presentation?
\end{question}

Downey and Melnikov observed that a positive solution of McNicholl's
conjecture would give a negative answer to
\Cref{question:computable-dual}: for a noncomputable right-c.e. $p\in(1,2)$, take a
computable copy $B$ of $L^p([0,1];\R)$.  Then $B$ is reflexive and,
by classical duality, $B^*\cong L^q([0,1];\R)$, where
$q=p/(p-1)>2$, so $B^*$ is separable.  If $B^*$ had a computable
presentation, \Cref{thm:mcnicholl-exponent-recovery}(i) would make
$q$, and hence $p=q/(q-1)$, computable, a contradiction. Our main theorem supplies such a
counterexample.

\begin{theorem}\label{thm:main}
Let $\F$ be either $\R$ or $\C$.  If $1<p<2$ is right-c.e., then
$L^p([0,1];\F)$ has a computable presentation.
\end{theorem}

The proof combines a locally
uniform effective stability theorem with a varying-exponent
Fra\"{\i}ss\'e-style construction. The purely atomic analogue behaves differently.  We prove the
following complementary result.

\begin{theorem}\label{thm:ellp-recovery}
Let $\F$ be either $\R$ or $\C$, and let $1<p<2$.  If
$\ell^p(\F)$ has a computable presentation, then $p$ is a computable
real.
\end{theorem}

The proof follows by
recovering $p$ from the decay of symmetric second differences of the
norm.  For $1<p<2$, the argument requires only a single atomic
coordinate and therefore extends to every separable $L^p$-space of
dimension at least two with a nonzero atomic part. Together with the known cases, this shows that
$\ell^p(\F)$ is computably presentable exactly when $p$ is computable,
for every $1\le p<\infty$.
Combining these results with \Cref{thm:main} and the classical
classification yields a complete classification of computable
presentability for separable $L^p$-spaces with $1\le p<\infty$.

As a complementary application of the Fra\"{\i}ss\'e-style finite-dimensional construction
method, we formulate an effective Fra\"{\i}ss\'e theorem for fixed
families of finite-dimensional Banach spaces with effectively
enumerable dense families of linear isometric embeddings.
Applying it to finite-dimensional
$\ell_\infty$-spaces gives a computable presentation of the
Gurarij Banach space.

The paper is organized as follows.  \Cref{sec:preliminaries} collects
the required background on computable analysis, measure theory and the geometry of Banach spaces.
\Cref{sec:banach-analysis} establishes the locally uniform effective
stability theorem, and \Cref{sec:proof-main} uses it to prove
\Cref{thm:main}.
\Cref{sec:exponent-classification} proves exponent recovery for
$\ell^p$, extends the argument to spaces with a nonzero atomic part,
and completes the classification of computably presentable separable
$L^p$-spaces.
Finally, \Cref{sec:gurarij-presentations} proves the effective
Fra\"{\i}ss\'e theorem for Banach spaces and applies it to the
Gurarij space.

\section{Preliminaries}\label{sec:preliminaries}
A \emph{computable Polish presentation} is a triple
$(X,d,(a_i)_{i\in\N})$, where $(X,d)$ is complete, $(a_i)$ is dense,
and $d(a_i,a_j)$ is a computable real uniformly in $i,j$.  The points
$a_i$ are called \emph{special points}.  A \emph{fast Cauchy name} for
$x\in X$ is a sequence $(i_n)$ such that $d(a_{i_n},a_{i_{n+1}})<2^{-n-1}$ and $x=\lim_n a_{i_n}$.
A map between computable Polish spaces is
\emph{computable} if a single Turing functional transforms every fast
Cauchy name of an input into a fast Cauchy name of its image.
For $i\in\N$ and $r\in\Q_{>0}$, let
$B(i,r)=\{x\in X:d(x,a_i)<r\}$ be a \emph{basic open ball}.  An open
set is \emph{c.e. open} if it is the union of a c.e. family of basic
open balls. The effective compactness principle used throughout the
proof is as follows.

\begin{lemma}[Downey and Melnikov {\cite{DMCC}}]
\label{thm:effective-compactness}
Let $X=(X,d,(a_i)_{i\in\N})$ be a nonempty computable Polish space.  The following are equivalent.
\begin{enumerate}[label=(\roman*)]
\item Given $n$, one can compute a finite tuple
$(i_j,r_j)_{j<m}$, where $r_j\in\Q_{>0}$ and $r_j<2^{-n}$, such that $X=\bigcup_{j<m}B(i_j,r_j)$.
\item There is a partial computable functional $\Gamma$ such that,
for every index $e$ enumerating a c.e. family $\mathcal U_e$ of basic
open balls, $\Gamma(e)$ is defined if and only if $\mathcal U_e$
covers $X$, and in that case $\Gamma(e)$ returns a finite subfamily
that still covers $X$.
\end{enumerate}
The conversions between witnesses to these conditions are uniform.
Any space satisfying these conditions is compact and is called
\emph{computably compact}.
\end{lemma}

\begin{fact}[\cite{DMCC}]
\label{fact:effective-compactness-closure}
The following operations are uniform.
\begin{enumerate}[label=(\roman*)]
\item Finite products of computably compact spaces are computably
compact.
\item If $X$ is nonempty and computably compact and $f:X\to\R$ is
computable, then the values $\max_X f$ and $\min_X f$ are computable.
\end{enumerate}
\end{fact}

We now specialize the preceding framework to Banach spaces.  Throughout
the paper, fix a scalar field $\F\in\{\R,\C\}$.  Put
$\Q_{\R}=\Q$ and $\Q_{\C}=\Q+i\Q$.

\begin{definition}[Pour-El and Richards \cite{PER}]\label{def:computable-banach-presentation}
A \emph{computable presentation} of a separable $\F$-Banach space
$B$ is a sequence $(b_n)_{n\in\N}$ with dense linear span such that
$\left\|\sum_{j<m}q_jb_{i_j}\right\|$
is a computable real, uniformly in $m$, the indices $i_j$, and the
coefficients $q_j\in\Q_{\F}$. 
\end{definition}
A fixed effective enumeration of all
finite $\Q_{\F}$-linear combinations of the $b_n$'s gives the special
points of the associated computable Polish presentation.
Hence the definition above is equivalent
to a computable Polish presentation in which vector addition
$B\times B\to B$ and scalar multiplication $\F\times B\to B$, with
the standard computable presentation of $\F$, are computable.

Unless explicitly stated, all finite-dimensional spaces
below have positive dimension.
Write $\mathcal L(E,F)$ for the bounded linear operators and put
$m(T)=\inf_{\|x\|=1}\|Tx\|$.
A \emph{linear isomorphic embedding} is a bounded linear map with
$m(T)>0$; then $m(T)=\|T^{-1}:T(E)\to E\|^{-1}$.
A linear map is an \emph{isometric embedding} if it preserves norms.
We write $\Emb(E,F)$ for the set of linear isometric embeddings
from $E$ to $F$, equipped with the operator-norm metric.
We write $\Iso(E)$ for the group of surjective linear isometries
from $E$ to itself.

 Call
$E=(\F^d,\|\cdot\|_E)$ \emph{effective} if its norm is computable
uniformly on rational vectors.  For nonzero $E$, one can uniformly
find positive rationals $a_E,b_E$ such that
\begin{equation}\label{eq:effective-norm-comparison}
 a_E\|x\|_2\leq\|x\|_E\leq b_E\|x\|_2.
\end{equation}
Indeed, effectively choose rationals $c_j>\|e_j\|_E$ for the standard
basis vectors.  The triangle and Cauchy--Schwarz inequalities give
$\|x\|_E
\leq\left(\sum_{j=1}^dc_j^2\right)^{1/2}\|x\|_2$.
Thus one can effectively choose a rational $b_E$ satisfying the upper
comparison bound; the same estimate is an effective Lipschitz bound for
$\|\cdot\|_E$.  Hence the minimum of $\|\cdot\|_E$ on the computably
compact Euclidean unit sphere is a computable real.  The minimum is
positive, so one can effectively choose a positive rational $a_E$
below it.

The lemma below is a standard consequence of effective
compactness. 
\begin{lemma}[Brattka
\cite{BrattkaUniform}; Downey and Melnikov \cite{DMCC}]\label{lem:effective-operator-extrema}
Let $E$ and $F$ be nonzero effective finite-dimensional spaces.  From
their norm algorithms and computable names for the matrix entries of
$T\in\mathcal L(E,F)$, both $\|T\|$ and $m(T)$
are uniformly computable. 
\end{lemma}

Let $E$ be $k$-dimensional, where $k\geq1$.  A basis $(x_i)_{i<k}$ of $E$ is an
\emph{Auerbach basis} if its biorthogonal functionals
$(x_i^*)_{i<k}\subseteq E^*$ satisfy $x_i^*(x_j)=\delta_{ij}$ and $\|x_i\|=\|x_i^*\|=1$ for all $i,j<k$.

\begin{lemma}[Auerbach's lemma; see Banach {\cite{Banach}}]\label{lem:auerbach}
Every nonzero finite-dimensional real or complex normed space $E$ has an Auerbach
basis.  For such a basis $(x_i)_{i<k}$,
\begin{equation}\label{eq:auerbach-coordinate-bounds}
 \max_{i<k}|a_i|
 \leq \left\|\sum_{i<k}a_ix_i\right\|
 \leq \sum_{i<k}|a_i|
 \qquad(a_0,\ldots,a_{k-1}\in\F).
\end{equation}
In particular, for every normed space $F$ and
$S,T\in\mathcal L(E,F)$,
\[
 \|S-T\|
 \leq\sum_{i<k}\|(S-T)x_i\|
 \leq k\max_{i<k}\|(S-T)x_i\|.
\]
\end{lemma}

\begin{proof}[Proof of the final assertion]
For $x=\sum_i a_ix_i$ with $\|x\|\leq1$,
\eqref{eq:auerbach-coordinate-bounds} gives $|a_i|\leq1$ for every $i$; apply the
triangle inequality and take the supremum over $x$.
\end{proof}
\begin{definition}
    Let $T:E\to F$ be a linear isomorphic embedding.  Its
\emph{logarithmic distortion} is
\[
 \ldis(T)=\sup_{x\neq0}
 \left|\log\frac{\|Tx\|}{\|x\|}\right|.
\]
For $\delta>0$, write $\Emb^{\log}_\delta(E,F)$ for the set of
linear isomorphic embeddings $T:E\to F$ with $\ldis(T)<\delta$.
\end{definition}

\begin{fact}\label{lem:distortion-calculus}
Let $S$ and $T$ be composable linear isomorphic embeddings. Then $\ldis(ST)\leq\ldis(S)+\ldis(T)$. Also $\ldis(T)=\max\{\log\|T\|,\log\|T^{-1}:T(E)\to E\|\}$, hence $\ldis(T^{-1})=\ldis(T)$.
\end{fact}

For two $k$-dimensional normed linear spaces $E$ and $F$, their
\emph{logarithmic Banach--Mazur distance} (defined in \cite{BanachMazur}) is
\begin{equation*}
 \distBM(E,F)=
 \inf_{T:E\cong F}\log\bigl(\|T\|\,\|T^{-1}\|\bigr),
\end{equation*}
where the infimum is over all linear isomorphisms from $E$ onto $F$.
With this logarithmic convention, $\distBM$ is a metric on the
linear-isometry classes of $k$-dimensional normed spaces.

\begin{fact}\label{lem:bm-versus-ldis}
For nonzero finite-dimensional $E$ and $F$ of the same dimension,
$$\distBM(E,F)=2\inf_{T:E\cong F}\ldis(T).$$
\end{fact}

\begin{proof}
For every isomorphism $T$, put
$c=(\|T^{-1}\|/\|T\|)^{1/2}$.  Then
$\ldis(cT)=\tfrac12\log(\|T\|\,\|T^{-1}\|)\leq\ldis(T)$;
take infima.
\end{proof}

The \emph{age} $\Age(E)$ of a Banach space $E$ is the collection of
linear-isometry classes of its nonzero finite-dimensional subspaces.  For $k\geq1$,
$\Age_k(E)$ is its $k$-dimensional part.

\begin{lemma}[{\cite[Example~2.31]{FLAMT}}]
\label{lem:lp-age-closed}
For every $1<r<2$ and $k\geq1$,
$\Age_k(L^r[0,1])$ is closed under Banach--Mazur limits.
\end{lemma}

\begin{definition}[see {\cite{FLAMT}}]\label{def:auh}
An infinite-dimensional separable Banach space $B$ is
\emph{approximately ultrahomogeneous} if, for every
$E\in\Age(B)$, every $\gamma,\eta\in\Emb(E,B)$,
and every $\epsilon>0$, there is
$U\in\Iso(B)$ such that $\|U\gamma-\eta\|<\epsilon$.
\end{definition}

\begin{lemma}[Lusky \cite{Lusky}]
\label{thm:lp-auh}
For every $1<r<2$, the space $L^r([0,1];\F)$ is approximately
ultrahomogeneous.
\end{lemma}

A measurable space is \emph{standard Borel} if its
$\sigma$-algebra is generated by a Polish topology.
A measurable set $A$ with $0<\mu(A)<\infty$ is an \emph{atom} if
every measurable $B\subseteq A$ has measure $0$ or $\mu(A)$; a
measure with no atoms is \emph{atomless}.
The \emph{pushforward} of $\mu$ under a measurable
map $T$ is $(T_*\mu)(E)=\mu(T^{-1}(E))$. 

\begin{lemma}[see
{\cite[Theorem~17.41]{Kechris}}]\label{lem:rokhlin-model}
Let $X$ be a standard Borel space carrying an atomless Borel
probability measure $\mu$.  Then, for every $1\leq r<\infty$,
$L^r(X,\mu;\F)$ is linearly isometric to $L^r([0,1];\F)$.
\end{lemma}

For $N\geq1$, partition $[0,1]$ into intervals
$(I_{N,j})_{j<N}$ of length $1/N$.  For $1\leq r<\infty$, the map
\begin{equation}\label{eq:normalized-step-embedding}
 (a_j)_{j<N}\mapsto
 N^{1/r}\sum_{j<N}a_j\mathbf1_{I_{N,j}}
\end{equation}
is a linear isometric embedding of $\ell_r^N$ into $L^r[0,1]$.
On $[0,1]$, $\lambda$ denotes Lebesgue measure.
\begin{lemma}[see {\cite[Proposition~6.1.3]{AlbiacKalton}}]\label{thm:dyadic-approximation}
Let $1\leq r<\infty$, let
$\mathcal D_n=\{I_{2^n,j}:j<2^n\}$, and define the dyadic
conditional expectation by
\[
 P_nf=2^n\sum_{A\in\mathcal D_n}
       \left(\int_A f\dd\lambda\right)\mathbf1_A.
\]
Then $P_n$ is a linear projection of norm $1$ on $L^r[0,1]$,
and its range is linearly isometric to $\ell_r^{2^n}$ via
\eqref{eq:normalized-step-embedding}.  For every finite-dimensional
subspace $E\subseteq L^r[0,1]$,
$\|(P_n-\id)\restr E\|\to 0$.
\end{lemma}

\begin{corollary}\label{cor:finite-lr-approximation}
For $1\leq r<\infty$, every $E\in\Age_k(L^r[0,1])$, and
every $\eta>0$, there are $N\geq k$ and a $k$-dimensional subspace
$F\subseteq\ell_r^N$ such that $\distBM(E,F)<\eta$.
\end{corollary}

\begin{proof}
By \Cref{thm:dyadic-approximation}, the restrictions $P_n\restr E$
are eventually injective and their norms and inverse norms tend
to $1$.  Identifying $P_n(E)$ with a subspace of
$\ell_r^{2^n}$ therefore gives the assertion.
\end{proof}

The real case of the following theorem is due to Bretagnolle,
Dacunha-Castelle, and Krivine \cite{BDK}; the complex case is due to
Herz \cite{Herz} and Rosenthal \cite{Rosenthal}. 
\begin{theorem}
\label{thm:stable-variable-embedding}
If $1<u\leq v\leq2$, then there is
$T\in\Emb(L^v([0,1];\F),L^u([0,1];\F))$.
In particular,
$\Age(L^v[0,1])\subseteq\Age(L^u[0,1])$.
\end{theorem}

We record the measure-theoretic form of the
Plotkin--Rudin equimeasurability principle.  The real form is due to
Plotkin \cite{Plotkin}, and the complex form to Rudin \cite{Rudin}.

Let $n\geq1$, let $1<r<\infty$, and let $\mu$ be a finite positive
Borel measure on $\F^n$.  We say that $\mu$ has a \emph{finite $r$-moment}
if $\int_{\F^n}\|w\|_2^r\dd\mu(w)<\infty$.
For such a measure,
its \emph{$r$-characteristic} is the function
\[
 \widehat\mu^{(r)}(s)
 =\left(\int_{\F^n}
   \left|1+\sum_{i=1}^n s_iw_i\right|^r\dd\mu(w)\right)^{1/r},
 \qquad s=(s_1,\ldots,s_n)\in\F^n.
\]

\begin{theorem}[Plotkin--Rudin]
\label{thm:plotkin-rudin-uniqueness}
Let $n\geq1$, and let
$1<r<\infty$ not be a positive even integer.  If $\mu$ and $\nu$
are finite positive Borel measures on $\F^n$ with finite $r$-moments and
$\widehat\mu^{(r)}(s)=\widehat\nu^{(r)}(s)$ for every $s\in\F^n$,
then $\mu=\nu$.
\end{theorem}

\section{An effective stability modulus}\label{sec:banach-analysis}

We obtain an effective strengthening of the qualitative approximate
ultrahomogeneity in \Cref{thm:lp-auh}: a computable stability modulus
for almost isometric embeddings, depending only on the dimension,
the desired accuracy, and a compact rational interval of exponents.
Throughout this section fix a rational interval
$I=[a,b]\subset(1,2)$.

\begin{theorem}
\label{thm:uniform-stability}
There is a total computable rational-valued function $\omega_I:\N_{\geq1}\times\Q_{>0}\to\Q_{>0}$
such that, for every $r\in I$, $\varepsilon \in \Q_{>0}$, $k \ge 1$, 
$E\in\Age_k(L^r[0,1])$, and
$\gamma_0,\gamma_1\in\Emb^{\log}_{\omega_I(k,\eps)}(E,L^r[0,1])$,
there is
$U\in\Iso(L^r[0,1])$ satisfying $\|U\gamma_0-\gamma_1\|<\eps$.
\end{theorem}
Note that the algorithm computes the modulus only: it does not compute $U$.
To control $\|U\gamma_0-\gamma_1\|$, it suffices to make
$U\gamma_0x_i$ close to $\gamma_1x_i$ on an Auerbach basis
$(x_i)_{i<k}$ of $E$.  The small distortions of the embeddings imply
that corresponding linear combinations of these two tuples have
nearly equal norms.  To use this norm information to make the two tuples coordinatewise close after applying a suitable surjective linear isometry to one of them, we associate to each tuple a measure on a compact
projective space whose cosine transform records the $r$th powers
of these norms.  Injectivity of the cosine transform and computable
compactness then give an effective bound on the Wasserstein distance
between the fixed-mass extensions of the two measures, uniform in
$r\in I$.  A coupling produces nearby isometric copies of the tuples
on an auxiliary atomless probability space, and \Cref{thm:lp-auh}
transfers this approximation to the original tuples under some
$U\in\Iso(L^r[0,1])$.

Let $Q_k$ be the projective space of $\F^k$.  Its points are
one-dimensional subspaces $[u]=\F u$, with $u$ chosen from the
Euclidean unit sphere $S(\F^k)$.  Put
$G_\F=\{\theta\in\F:|\theta|=1\}$, so that
$G_\R=\{-1,1\}$ and $G_\C$ is the unit circle.
Equip $Q_k$ with
\[
 d_Q([u],[v])=\min_{\theta\in G_\F}\|u-\theta v\|_2.
\]
The identity
\[
\begin{aligned}
d_Q([u],[v])^2
&=\min_{\theta\in G_\F}
  \left(
    \|u\|_2^2+\|\theta v\|_2^2
    -2\operatorname{Re}
      \langle u,\theta v\rangle_{\ell_2^k}
  \right)\\
&=2-2\max_{\theta\in G_\F}
  \operatorname{Re}
  \langle u,\theta v\rangle_{\ell_2^k}=2-2\left|
  \langle u,v\rangle_{\ell_2^k}
  \right|
\end{aligned}
\]
gives $\operatorname{diam}(Q_k)\leq\sqrt2$.
Take as special points the classes of unit vectors in $\Q_\F^k$.
It is clear from \Cref{thm:effective-compactness}(i) that $Q_k$ is
computably compact.

For a finite Borel measure $\sigma$ on $Q_k$ and $r>0$, define \emph{the projective cosine transform}
\[
 (\mathcal C_r\sigma)(z)
   =\int_{Q_k}|z\cdot u|^r\dd\sigma([u]),
 \qquad z\in\F^k,
\]
where $z\cdot u=\sum_{j<k}z_ju_j$.  The absolute value makes the
integrand independent of the unit representative of $[u]$.
The following lemma is a consequence of
\Cref{thm:plotkin-rudin-uniqueness}. 

\begin{lemma}\label{lem:cosine-injective}
For finite positive Borel measures \(\sigma\) and \(\nu\) on \(Q_k\), if
\(1<r<2\) and \(\mathcal C_r\sigma=\mathcal C_r\nu\), then
\(\sigma=\nu\).
\end{lemma}

\begin{proof}
The space \(Q_1\) is a singleton, so a finite measure on
it is determined by its total mass.
Hence assume \(k\geq2\).
Put \(m=\sigma+\nu\). We may assume that \(m(Q_k)>0\).
  
  We first choose a nonzero functional
\(c\in(\F^k)^*\) such that $m(H_c)=0$ where
\[
 H_c=\{[u]\in Q_k:c(u)=0\}.
\]
Indeed, let \(B\) be the Euclidean unit ball in the
underlying real vector space of \((\F^k)^*\), and let
\(\mathcal L_B\) be Lebesgue measure restricted to \(B\).  The set
$\{(c,[u])\in B\times Q_k:c(u)=0\}$
is Borel.  For each fixed nonzero \(u\), the section
\(\{c\in B:c(u)=0\}\) is contained in a proper real linear subspace.  It is therefore \(\mathcal L_B\)-null.  The Fubini-Tonelli theorem
now yields
\[
 \int_B m(H_c)\dd\mathcal L_B(c)
 =\int_{Q_k}\mathcal L_B\{c\in B:c(u)=0\}\dd m([u])=0.
\]
Thus \(m(H_c)=0\) for almost every \(c\in B\); choose one such \(c\).
It is automatically nonzero.  In particular,
\(\sigma(H_c)=\nu(H_c)=0\).
Complete \(c\) to a basis
\[c,d_1,\ldots,d_{k-1}\] of the \(\F\)-linear dual
\((\F^k)^*\).  On \(Q_k\setminus H_c\), define the affine-coordinate
map
\[
 \chi_c([u])
 =\left(\frac{d_i(u)}{c(u)}\right)_{1\leq i<k}
 \in\F^{k-1}.
\]
For $v \in \F^k$, define
$T(v)=(c(v),d_1(v),\ldots,d_{k-1}(v))$,
then 
$w\mapsto
 \left[
 \frac{T^{-1}(1,w_1,\ldots,w_{k-1})}
      {\|T^{-1}(1,w_1,\ldots,w_{k-1})\|_2}
 \right]$
is the inverse of \(\chi_c\).  Hence \(\chi_c\) is a homeomorphism from \(Q_k\setminus H_c\) onto \(\F^{k-1}\).

Define finite measures on \(Q_k\setminus H_c\) by $\dd\rho_\sigma([u])=|c(u)|^r\dd\sigma([u])$, and 
$\dd\rho_\nu([u])=|c(u)|^r\dd\nu([u])$,
and let 
$$\alpha=(\chi_c)_*\rho_\sigma,\quad \beta=(\chi_c)_*\rho_\nu.$$
Since $1<r<2$, we
have $\|w\|_2^r\leq\sum_{i=1}^{k-1}|w_i|^r$ for $w \in \F^{k-1}$.  Hence 
\[
 \begin{aligned}
 \int_{\F^{k-1}}\|w\|_2^r\dd\alpha(w)
 \leq\sum_{i=1}^{k-1}\int_{\F^{k-1}}|w_i|^r\dd\alpha(w)
 &=\sum_{i=1}^{k-1}\int_{Q_k\setminus H_c}
   \left|\frac{d_i(u)}{c(u)}\right|^r|c(u)|^r
   \dd\sigma([u])\\
 &=\sum_{i=1}^{k-1}
   \int_{Q_k\setminus H_c}|d_i(u)|^r\dd\sigma([u])<\infty.
 \end{aligned}
\]
The same calculation applies to
\(\beta\).  Thus both measures have finite \(r\)-moments.

Fix \(s=(s_i)_{1\leq i<k}\in\F^{k-1}\), and set
$\ell_s=c+\sum_{i=1}^{k-1}s_id_i\in(\F^k)^*$.  If
$(e_j)_{j<k}$ is the standard basis of $\F^k$, put $z_s=(\ell_s(e_j))_{j<k}$.
Then $\F$-linearity gives
$\ell_s(u)=\sum_{j<k}\ell_s(e_j)u_j=z_s\cdot u$ for every
$u\in\F^k$.  Then the $r$-characteristic satisfies
\[
 \begin{aligned}
 \bigl(\widehat\alpha^{(r)}(s)\bigr)^r
 &=\int_{Q_k\setminus H_c}
   \left|1+\sum_{i=1}^{k-1}s_i\frac{d_i(u)}{c(u)}\right|^r
   |c(u)|^r\dd\sigma([u])\\
 &=\int_{Q_k\setminus H_c}|\ell_s(u)|^r\dd\sigma([u])
 =\int_{Q_k}|z_s\cdot u|^r\dd\sigma([u])
   =(\mathcal C_r\sigma)(z_s).
 \end{aligned}
\]
The third equality uses \(\sigma(H_c)=0\).  The identical computation
with \(\nu\) gives
$\bigl(\widehat\beta^{(r)}(s)\bigr)^r
 =(\mathcal C_r\nu)(z_s)$.
The hypothesis \(\mathcal C_r\sigma=\mathcal C_r\nu\) therefore
implies
\(\widehat\alpha^{(r)}=\widehat\beta^{(r)}\).  Since \(1<r<2\),
\Cref{thm:plotkin-rudin-uniqueness} yields \(\alpha=\beta\).
Since \(\chi_c\) is a homeomorphism, $\rho_\sigma=(\chi_c^{-1})_*\alpha$ and $\rho_\nu=(\chi_c^{-1})_*\beta$,
hence \(\rho_\sigma=\rho_\nu\).  To remove the weight,
put $E_n=\{[u]\in Q_k:|c(u)|\geq1/n\}$.
For every Borel \(A\subseteq Q_k\), the function \(|c|^{-r}\) is
bounded on \(E_n\), and hence
\[
 \sigma(A\cap E_n)
 =\int_{A\cap E_n}|c|^{-r}\dd\rho_\sigma
 =\int_{A\cap E_n}|c|^{-r}\dd\rho_\nu
 =\nu(A\cap E_n).
\]
Finally, \(E_n\uparrow Q_k\setminus H_c\), while both measures vanish
on \(H_c\).  Continuity from below gives \(\sigma(A)=\nu(A)\) for
every Borel \(A\), and therefore \(\sigma=\nu\).
\end{proof}

Adjoin to $Q_k$ an isolated \emph{cemetery point} $\partial$: define $K_k=Q_k\sqcup\{\partial\}$, and set $d(\partial,q)=1$ for $q\in Q_k$ and $d(p,q)=d_Q(p,q)$ for $p,q \in Q_k$.
Then $(K_k,d)$ is a
computably compact metric space.  For $M>0$, let
$\mathcal P_M(K_k)$ be the space of finite Borel measures on $K_k$
having total mass $M$, equipped with the 
\emph{$1$-Wasserstein metric}
$W_1(\mu,\nu)=\inf_{\pi\in\Pi(\mu,\nu)}
       \int d(x,y)\dd\pi(x,y)$,
where $\Pi(\mu,\nu)$ is the set of finite Borel measures on
$K_k\times K_k$ whose first and second marginals are respectively
$\mu$ and $\nu$. 
The point $\partial$ is used only to supplement two measures to a common
total mass. 

\begin{lemma}[Marcone and
Valenti \cite{MarconeValenti}; see Downey and Melnikov
{\cite[Lemma~4.50]{DMCC}}]\label{lem:measure-compact}
$\mathcal P_M(K_k)$ is computably compact, uniformly in $k$ and
rational $M>0$.  More precisely, let $(x_j)_{j\in\N}$ consist of
$\partial$ and the classes of unit vectors in
$\Q_\F^k$.  The rational atomic measures
\begin{equation}\label{eq:atomic-special-measures}
 \mathscr A_{k,M}=
 \left\{\sum_{j<\ell}a_j\delta_{x_{i_j}}:
 \ell\geq1,\ a_j\in\Q_{\geq0},\ \sum_{j<\ell}a_j=M\right\}
\end{equation}
form an effective special sequence, with uniformly computable finite
basic covers.
\end{lemma}

For $r\in I$, $z\in S(\F^k)$, define $$\phi_{r,z}([u])=|z\cdot u|^r \text{ and } \phi_{r,z}(\partial)=0.$$ 
The map $(r,z,q)\mapsto\phi_{r,z}(q)$ is computable. Since
$|s^r-t^r|\leq b|s-t|$ for all $s,t \in [0,1]$,
we have
\[
 |\phi_{r,z}(q)-\phi_{r,z}(q')|\leq b\,d(q,q').
\]

A fast $W_1$-Cauchy name for $\mu\in\mathcal P_M(K_k)$ supplies
$\mu_n\in\mathscr A_{k,M}$ with $W_1(\mu_n,\mu)<2^{-n}$.
Fix any coupling $\pi\in\Pi(\mu_n,\mu)$. The preceding Lipschitz estimate gives
\[
 \begin{aligned}
 \left|\int_{K_k}\phi_{r,z}\dd\mu_n
       -\int_{K_k}\phi_{r,z}\dd\mu\right|
 &=\left|\int_{K_k\times K_k}
       \bigl(\phi_{r,z}(q)-\phi_{r,z}(q')\bigr)
       \dd\pi(q,q')\right|\\
 &\leq b\int_{K_k\times K_k}d(q,q')\dd\pi(q,q').
 \end{aligned}
\]
Taking the
infimum over all $\pi\in\Pi(\mu_n,\mu)$ yields
\[
 \left|\int_{K_k}\phi_{r,z}\dd\mu_n
       -\int_{K_k}\phi_{r,z}\dd\mu\right|
 \leq bW_1(\mu_n,\mu)<b2^{-n}.
\]
Since $\mu_n\in\mathscr A_{k,M}$, the first integral is a finite sum
computable uniformly in $r,z,\mu_n$.  The displayed error bound is
independent of $r$ and $z$.  Consequently the map
$(r,z,\mu)\mapsto\int_{K_k}\phi_{r,z}\dd\mu$
is computable uniformly in $k$, $M$, $r$, $z$, and $\mu$, and hence so is
$(r,z,\mu,\nu)\mapsto
 \int_{K_k}\phi_{r,z}\dd(\mu-\nu)$.
Define
\[
 D(r,\mu,\nu)=
 \sup_{\|z\|_2=1}
 \left|\int_{K_k}\phi_{r,z}\dd(\mu-\nu)\right|=\sup_{\|z\|_2=1}\left|  (\mathcal C_r\mu|_{Q_k})(z)-(\mathcal C_r\nu|_{Q_k})(z)\right|.
\]

\begin{lemma}\label{prop:D-computable}
For every $k\geq1$ and rational $M>0$, $D:I\times\mathcal P_M(K_k)^2\to\R$
is computable uniformly in $k,M$.
\end{lemma}

\begin{proof}
By the preceding integral computation,
$|\int\phi_{r,z}\,\dd(\mu-\nu)|$ is computable uniformly.  Since $S(\F^k)$ is uniformly computably compact,
its maximum $D$ is computable by
\Cref{fact:effective-compactness-closure}(ii).
\end{proof}

\begin{lemma}\label{lem:effective-inverse}
Given $k\geq1$ and positive rationals $M,\eta$, one can compute
a positive rational $\tau=\tau(I,k,M,\eta)$ such that
$D(r,\mu,\nu)<\tau$ implies $W_1(\mu,\nu)<\eta$
for all $r\in I$ and $\mu,\nu\in\mathcal P_M(K_k)$.
\end{lemma}

\begin{proof}
If $D(r,\mu,\nu)=0$, the cosine transforms of $\mu|_{Q_k}$ and
$\nu|_{Q_k}$ agree on $S(\F^k)$.  For every finite measure $\sigma$
on $Q_k$ and $z\ne0$,
$(\mathcal C_r\sigma)(z)
 =\|z\|_2^r(\mathcal C_r\sigma)
   \left(\frac{z}{\|z\|_2}\right)$.
Thus the transforms agree on all of $\F^k$, and
\Cref{lem:cosine-injective} gives $\mu|_{Q_k}=\nu|_{Q_k}$.
Their common total mass $M$ then gives $\mu=\nu$.
Consequently the computable function
\[
 H(r,\mu,\nu)=D(r,\mu,\nu)
              +\max\{\eta-W_1(\mu,\nu),0\}
\]
is strictly positive on the computably compact space
$I\times\mathcal P_M(K_k)^2$.  Its minimum $m$ is therefore positive
and computable by \Cref{fact:effective-compactness-closure}(ii).
Compute a rational $0<\tau<m$.  Whenever $W_1(\mu,\nu)\geq\eta$,
we have $D(r,\mu,\nu)=H(r,\mu,\nu)\geq m>\tau$.
\end{proof}

We now associate with each tuple in $L^r[0,1]^k$ a finite measure
that records its projective directions and their $L^r$-mass.
Let $f=(f_0,\ldots,f_{k-1})\in L^r[0,1]^k$.  Choose Borel
representatives, write $\mathbf f(t)=(f_0(t),\ldots,f_{k-1}(t))$,
and put $A_f=\{t\in[0,1]:\mathbf f(t)\ne0\}$.
The \emph{projective spectral measure} $\sigma_f$ is the pushforward
of $\|\mathbf f(t)\|_2^r\dd\lambda(t)$ on $A_f$ under the Borel map $t\mapsto\left[\frac{\mathbf f(t)}{\|\mathbf f(t)\|_2}\right]$.
Changing the Borel representatives does not change $\sigma_f$.
Since $1<r<2$,
\[
 \|\sigma_f\|=\sigma_f(Q_k)
 =\int_{A_f}\|\mathbf f(t)\|_2^r\dd\lambda(t)
 \leq\sum_{i<k}\|f_i\|_r^r<\infty.
\]
For every $z\in\F^k$, the definition gives
\begin{equation}\label{eq:spectral-norm}
 (\mathcal C_r\sigma_f)(z)
 =\int_{A_f}|z\cdot\mathbf f(t)|^r\dd\lambda(t)
 =\left\|\sum_{i<k}z_if_i\right\|_r^r.
\end{equation}

Given $M>0$ with $\|\sigma_f\|\leq M$, define its \emph{fixed-mass
cemetery extension} by
\[
 \widehat\sigma_f^{\,M}
 =\sigma_f+(M-\|\sigma_f\|)\delta_\partial
 \in\mathcal P_M(K_k),
\]
where $\delta_\partial$ is the Dirac measure at $\partial$.

\begin{lemma}
\label{lem:quantitative-coupling}
Let $k\geq1$, $1<r<2$, $M,\eta>0$, 
$f,g\in L^r[0,1]^k$, and suppose
$\|\sigma_f\|,\|\sigma_g\|\leq M$.  If
$W_1(\widehat\sigma_f^{\,M},\widehat\sigma_g^{\,M})<\eta$,
then there is $U\in\Iso(L^r[0,1])$ such that
$\max_{i<k}\|Uf_i-g_i\|_r^r<2\eta$.
\end{lemma}

\begin{proof}
Choose a coupling $\pi$ of the two cemetery extensions such that
\[
 c=\int d(x,y)\,\dd\pi(x,y)<\eta.
\]
Choose a Borel section $s:Q_k\to S(\F^k)$ by selecting the unit
representative whose first nonzero coordinate is positive real,
and extend it to $K_k$ by $s(\partial)=0$.
For $x,y\in K_k$, put
\[
 \alpha(x,y)=\sum_{i<k}s_i(x)\overline{s_i(y)},\qquad
 \zeta(x,y)=
 \begin{cases}
  \overline{\alpha(x,y)}/|\alpha(x,y)|,&\alpha(x,y)\ne0,\\
  1,&\alpha(x,y)=0.
 \end{cases}
\]
This Borel phase belongs to $G_\F$.  If $x,y\in Q_k$, then
\[
 \|\zeta(x,y)s(x)-s(y)\|_2^2
 =2-2|\alpha(x,y)|=d_Q(x,y)^2.
\]
Since $d_Q\leq\sqrt2<2$, it follows that
\begin{equation}\label{eq:coupling-pointwise}
 \|\zeta(x,y)s(x)-s(y)\|_2^r
 \leq 2^{r-1}d(x,y)
 \qquad(x,y\in K_k).
\end{equation}

The space
\[
 (\Omega,\mathbb P)
 =\bigl(K_k\times K_k\times[0,1],\,M^{-1}\pi\otimes\lambda\bigr)
\]
is a standard Borel probability space, and the Lebesgue factor
makes its measure atomless.  On this space define
\[
 F_i(x,y,t)=M^{1/r}\zeta(x,y)s_i(x),\qquad
 G_i(x,y,t)=M^{1/r}s_i(y)
 \qquad(i<k).
\]
By \Cref{lem:rokhlin-model}, identify $L^r(\Omega,\mathbb P)$
with $L^r[0,1]$, using the same identification for all $F_i,G_i$.
This preserves the norms of all linear combinations and the distances
$\|F_i-G_i\|_r$, allowing us to apply the approximate ultrahomogeneity
of $L^r[0,1]$ below.
For every $z\in\F^k$, the first marginal of $\pi$ and
\eqref{eq:spectral-norm} give
\[
 \begin{aligned}
 \left\|\sum_{i<k}z_iF_i\right\|_r^r
 =\int_{K_k\times K_k}|z\cdot s(x)|^r\,\dd\pi(x,y)
 =\int_{K_k}|z\cdot s(x)|^r\,\dd\widehat\sigma_f^{\,M}(x)
 =\left\|\sum_{i<k}z_if_i\right\|_r^r.
 \end{aligned}
\]
The second marginal gives the analogous identity for $G$ and $g$.
Then
\eqref{eq:coupling-pointwise} yields
\[
 \begin{aligned}
 \|F_i-G_i\|_r^r
 &=\int_{K_k\times K_k}
   |\zeta(x,y)s_i(x)-s_i(y)|^r\,\dd\pi(x,y)\\
 &\leq\int_{K_k\times K_k}
   \|\zeta(x,y)s(x)-s(y)\|_2^r\,\dd\pi(x,y)
\leq 2^{r-1}c.
 \end{aligned}
\]

The norm identities show that $f_i\mapsto F_i$ extends to a
linear isometric embedding of $\operatorname{span}\{f_i:i<k\}$
into $L^r[0,1]$, and similarly for $g_i\mapsto G_i$.
Consequently \Cref{thm:lp-auh} implies that, for every $u>0$,
there are $V,W\in\Iso(L^r[0,1])$ with
\[
 \max_{i<k}\|Vf_i-F_i\|_r<u,
 \qquad
 \max_{i<k}\|Wg_i-G_i\|_r<u.
\]

Put $\rho=(2^{r-1}c)^{1/r}$.  Since $c<\eta$ and $r<2$, we have
$\rho<(2\eta)^{1/r}$.  Choose $u>0$ such that
$2u+\rho<(2\eta)^{1/r}$ and set $U=W^{-1}V$.  Then, for every $i<k$,
\[
 \begin{aligned}
 \|Uf_i-g_i\|_r
 &=\|Vf_i-Wg_i\|_r\\
 &\leq\|Vf_i-F_i\|_r+\|F_i-G_i\|_r+\|G_i-Wg_i\|_r
 <2u+\rho<(2\eta)^{1/r}.
 \end{aligned}
\]
Thus $U\in\Iso(L^r[0,1])$ has the required property.
\end{proof}

\begin{proof}[Proof of \Cref{thm:uniform-stability}]
Given $k\geq1$ and $\eps\in\Q_{>0}$, put
\[
 s=\min\{1,\eps/k\},\qquad \eta=s^2/8,
\]
and use \Cref{lem:effective-inverse} with $M=4k$ to compute
$\tau>0$ such that $D<\tau$ implies $W_1<\eta$.  Define
\[
 \delta:=\min\{1,\tau/(10k)\},\qquad
 \omega_I(k,\eps):=\delta/2.
\]

Fix $r\in I$, $E\in\Age_k(L^r[0,1])$, and
$\gamma_0,\gamma_1\in\Emb^{\log}_{\omega_I(k,\eps)}(E,L^r[0,1])$.
Since $0<\delta\leq1$ and $\delta/2<\log(1+\delta)$, for
$j=0,1$ and every $x\in E$ we have
\[
 (1+\delta)^{-1}\|x\|\leq\|\gamma_jx\|
 \leq(1+\delta)\|x\|.
\]
Choose an Auerbach basis $(x_i)_{i<k}$ and put
$f_i=\gamma_0x_i$, $g_i=\gamma_1x_i$.  Since the Auerbach basis
is normalized, $\|x_i\|_E=1$, so the upper bound on $\gamma_0$
gives $\|f_i\|_r\leq1+\delta$ for every $i<k$.
Moreover, $0<r/2<1$, so subadditivity of $u\mapsto u^{r/2}$
on $[0,\infty)$ gives, pointwise, $\|\mathbf f(t)\|_2^r
 =\left(\sum_{i<k}|f_i(t)|^2\right)^{r/2}
 \leq\sum_{i<k}|f_i(t)|^r$.
Integrating and using the definition of $\sigma_f$, we obtain
\[
 \begin{aligned}
 \|\sigma_f\|
 =\int_{A_f}\|\mathbf f(t)\|_2^r\,\dd\lambda(t)
 \leq\sum_{i<k}\int_0^1|f_i(t)|^r\,\dd\lambda(t)
 \leq k(1+\delta)^r\leq k2^r\leq4k,
 \end{aligned}
\]
where the last two inequalities use $\delta\leq1$ and $r<2$.
The same argument applies to $g$, using the upper bound on $\gamma_1$.
Thus their cemetery extensions of mass $4k$ are defined.

For $z\in S(\F^k)$, put $x_z=\sum_i z_ix_i$.  By
\Cref{lem:auerbach}, $\|x_z\|_E\leq\sum_i|z_i|\leq\sqrt{k}$.
By \eqref{eq:spectral-norm} and the bounds on $\gamma_0,\gamma_1$,
both $(\mathcal C_r\sigma_f)(z)$ and $(\mathcal C_r\sigma_g)(z)$
lie between $(1+\delta)^{-r}\|x_z\|_E^r$ and
$(1+\delta)^r\|x_z\|_E^r$.  Since $\phi_{r,z}(\partial)=0$ and
$1<r<2$, taking the supremum over $z\in S(\F^k)$ gives
\[
 \begin{aligned}
 D(r,\widehat\sigma_f^{\,4k},\widehat\sigma_g^{\,4k})
 &\leq k^{r/2}\bigl((1+\delta)^r-(1+\delta)^{-r}\bigr)\\
 &\leq k\bigl((1+\delta)^2-(1+\delta)^{-2}\bigr)\\
 &\leq5k\delta\leq\tau/2<\tau,
 \end{aligned}
\]
here $(1+\delta)^2-1\leq3\delta$ and
$1-(1+\delta)^{-2}\leq2\delta$ for $0\leq\delta\leq1$.
The inverse modulus and \Cref{lem:quantitative-coupling} therefore
give $U\in\Iso(L^r[0,1])$ with
$\|Uf_i-g_i\|_r^r<2\eta$ for all $i<k$.  Since
$2\eta=s^2/4<1$ and $r<2$,
\[
 \|U\gamma_0-\gamma_1\|
 \leq k\max_{i<k}\|Uf_i-g_i\|_r
 <k(2\eta)^{1/r}
 \leq k\sqrt{2\eta}=ks/2<\eps.
\]
This proves the theorem.
\end{proof}

\section{An effective varying-exponent Fra\"{\i}ss\'e-style construction}
\label{sec:proof-main}

Fix a right-c.e. real $p\in(1,2)$, rationals $1<a<p<b<2$, and a
computable nonincreasing sequence of rationals $q_m\downarrow p$ in
$I=[a,b]$.  Using the stability modulus from
\Cref{sec:banach-analysis}, we construct
$X_m=\ell_{q_m}^{d_m}$ with rational bonding maps
$I_m:X_m\to X_{m+1}$.  Here an operator between spaces with specified
coordinate bases is \emph{rational} if its matrix has entries in
$\Q_\F$.  Summable distortions will give a computable limiting norm
and a completion $B$.

There are two kinds of requirements.  Copy requirements embed each member of a
Banach--Mazur dense family $(D_e)_e\subseteq\Age(L^p[0,1])$ into
$X_m$ with vanishing distortion as $m\to\infty$.  Pair requirements supply,
for sufficiently undistorted $\gamma,\eta:D_e\to X_m$, an alternative
map $J:X_m\to X_{m+1}$ with $J\gamma\approx I_m\eta$.
We use these
alignments to build an approximate intertwining between a cofinal
subsequence of the renormed spaces $\bar X_m$ and finite-dimensional subspaces
of $L^p[0,1]$.  Its limit is a surjective linear isometry
$B\to L^p[0,1]$.
We adapt the Fra\"{\i}ss\'e-sequence construction in
\cite[proof of Theorem~2.26]{FLAMT} to varying exponents, using effective searches and summable distortions to obtain a computable limiting norm.
The next lemma aligns finitely many pairs while embedding a common
base space into the next exponent.

\begin{lemma}
\label{lem:finite-pair-package}
Let $r',r\in I\cap\Q$ with $r'\leq r$, and let $Y$ and a family $(E_i)_{1\leq i\leq s}$ be nonzero effective
finite-dimensional spaces in $\Age(L^r[0,1])$.
Fix positive rationals $\theta,\xi_i,\beta_i$ with
$\beta_i<\omega_I(\dim E_i,\xi_i/2)$.
One can effectively find $N\geq1$, a rational
operator $\Phi\in\Emb^{\log}_{\theta}(Y,\ell_{r'}^N)$,
and finite sets $\mathscr J_i$ of rational operators in
$\Emb^{\log}_{\theta}(Y,\ell_{r'}^N)$ such that for all $\gamma,\eta\in\Emb^{\log}_{\beta_i}(E_i,Y)$, there is $J\in\mathscr J_i$ such that $\|J\gamma-\Phi\eta\|<\xi_i$.
\end{lemma}

\begin{proof}
Put $\omega_i=\omega_I(\dim E_i,\xi_i/2)$ and
$\widehat\beta_i=(\beta_i+\omega_i)/2$.  Choose a rational
$v>e^\theta$ and set $\rho_i=\xi_i/(4v)$.
We first find finitely many rational centers whose open
$\rho_i$-balls cover
\[
 K_i=\{T\in\mathcal L(E_i,Y):
           \|T\|\leq e^{\beta_i},\ m(T)\geq e^{-\beta_i}\}.
\]

For matrices over $\F$, a \emph{rational coordinate box} is a product
of bounded closed intervals with rational endpoints in the matrix
coordinates.  When $\F=\mathbb C$, the real and imaginary parts of
each entry are treated as separate coordinates.  It is clear by \Cref{thm:effective-compactness}(i) that rational
coordinate boxes are computably compact in the operator-norm metric.
Choose a rational coordinate box $\mathscr B_i\supseteq K_i$;
this is effective since every matrix entry of $T\in K_i$ has modulus
at most $e^{\beta_i}b_{E_i}/a_Y$ by \eqref{eq:effective-norm-comparison}.
The computability of $\|T\|$ and $m(T)$, given by
\Cref{lem:effective-operator-extrema}, makes
$\mathscr B_i\setminus K_i$ a c.e. open subset of $\mathscr B_i$.
Enumerate finite sets $\mathcal N_i$ of rational centers in
$\Emb^{\log}_{\widehat\beta_i}(E_i,Y)$ such that 
$\mathscr B_i\subseteq
 (\mathscr B_i\setminus K_i)
 \cup\bigcup_{c\in\mathcal N_i}B(c,\rho_i)$
using \Cref{thm:effective-compactness}(ii).
This search terminates: $K_i$ is compact and lies in the open set
$\Emb^{\log}_{\widehat\beta_i}(E_i,Y)$, whose rational points are
dense. 

Having found these nets, search over $N$ and rational
matrices $\Phi,J_{i,c,d}:Y\to\ell_{r'}^N$ in $\Emb^{\log}_{\theta}(Y,\ell_{r'}^N)$ such that
$\|J_{i,c,d}c-\Phi d\|<\xi_i/2
 \quad(c,d\in\mathcal N_i)$.
All these conditions are semidecidable.  To prove termination,
choose $h\in\Emb(Y,L^r[0,1])$.  By
\Cref{thm:uniform-stability}, for each $c,d\in\mathcal N_i$ there is
$U_{i,c,d}\in\Iso(L^r[0,1])$ with $\|U_{i,c,d}hc-hd\|<\xi_i/2$.
Compose $h$ and the finitely many $U_{i,c,d}h$ with
$S\in\Emb(L^r[0,1],L^{r'}[0,1])$, supplied by
\Cref{thm:stable-variable-embedding}.
Let $W$ be the span of their ranges.  A sufficiently fine dyadic
conditional expectation $P$ satisfies
$\|(P-\id)\restr W\|<1-e^{-\theta/2}$ by
\Cref{thm:dyadic-approximation}.
Put $\delta=\|(P-\id)\restr W\|$.  Each of the maps
$V=Sh$ and $V=SU_{i,c,d}h$ belongs to $\Emb(Y,W)$.  Since $P$ is contractive, $\|PV\|\leq1$.  Moreover, for every
$x\in Y$ we have
\[
 \begin{aligned}
  \|PVx\|&\geq\|Vx\|-\|(P-\id)Vx\|
           \geq(1-\delta)\|x\|.
 \end{aligned}
\]
Thus $PV$ is an isomorphic embedding and
$\ldis(PV)\leq-\log(1-\delta)<\theta/2$.  Moreover, the contractivity of $P$
and the isometricity of $S$ give
\[
 \|PSU_{i,c,d}hc-PShd\|
 \leq\|U_{i,c,d}hc-hd\|<\xi_i/2.
\]
Identify $\operatorname{ran}P$ isometrically with a
finite $\ell_{r'}$-space by
\Cref{thm:dyadic-approximation}.  These maps satisfy the proposed certificate with
strict margins, so sufficiently close rational matrices do as well.

Set $\mathscr J_i=\{J_{i,c,d}:c,d\in\mathcal N_i\}$.  Given
$\gamma,\eta\in\Emb^{\log}_{\beta_i}(E_i,Y)$, choose net centers
$c,d$ within $\rho_i$ of $\gamma,\eta$.  Then
$\|J_{i,c,d}\gamma-\Phi\eta\|
 <2e^\theta\rho_i+\xi_i/2<\xi_i$.
\end{proof}

We now specify the requirements and construct the limit.
Enumerate all tuples $(s,k,N,A)$ with $s\in\N$, $1\leq k\leq N$,
and $A\in M_{N\times k}(\Q_\F)$ of rank $k$, arranging that
$s(e)\leq e$.  Set
\begin{equation}\label{eq:De-definition}
 D_e=\left(\F^{k(e)},\ x\mapsto
          \|A_ex\|_{\ell_{q_{s(e)}}^{N(e)}}\right).
\end{equation}
This is a uniformly effective family of finite-dimensional spaces.

\begin{lemma}\label{lem:D-dense}
Every $D_e$ belongs to $\Age(L^p[0,1])$, and the family
$(D_e)_e$ is Banach--Mazur dense in every fixed-dimensional
part of $\Age(L^p[0,1])$.
\end{lemma}

\begin{proof}
Membership follows from the normalized-indicator embedding
\eqref{eq:normalized-step-embedding} and
\Cref{thm:stable-variable-embedding}.
For density, write
$E_A^t=(\F^k,x\mapsto\|Ax\|_{\ell_t^N})$ for a full-rank matrix
$A\in M_{N\times k}(\F)$.  Given $E\in\Age_k(L^p)$ and $\eta>0$,
\Cref{cor:finite-lr-approximation} and a sufficiently small rational
perturbation of a basis matrix give a full-rank rational $A$ with
$\distBM(E,E_A^p)<\eta/2$. For $q_s \ge p$, the standard comparison of finite-dimensional $l_t$-norms gives, for every $x\in \F^k$,
\[
 \|Ax\|_{q_s}\leq\|Ax\|_p
 \leq N^{1/p-1/q_s}\|Ax\|_{q_s}.
\]Hence
$\distBM(E_A^{q_s},E_A^p)
 \leq(1/p-1/q_s)\log N\to 0$.
This proves the density.
\end{proof}

Define
\begin{equation}\label{eq:beta-choice}
 \beta_{k,h}=\frac12\min\{1,\omega_I(k,2^{-h-4})\}.
\end{equation}
Starting with $d_0=1$, we recursively construct
$X_m=\ell_{q_m}^{d_m}$ and rational bonding maps
$I_m:X_m\to X_{m+1}$.

Form the finite direct sum
\[
 Y_m=X_m\oplus_{q_m}D_0\oplus_{q_m}\cdots\oplus_{q_m}D_m,
\]
with its concatenated coordinate basis and norm
\[
 \|(x,u_0,\ldots,u_m)\|_{Y_m}
 =\left(\|x\|_{X_m}^{q_m}
       +\sum_{e=0}^{m}\|u_e\|_{D_e}^{q_m}\right)^{1/q_m}.
\]
This norm is computable uniformly in $m$.

To see that $Y_m\in\Age(L^{q_m}[0,1])$, first note that, for
$e\leq m$, the map $x\mapsto A_ex$ followed by
\eqref{eq:normalized-step-embedding} embeds $D_e$ isometrically into
$L^{q_{s(e)}}[0,1]$.  Since $s(e)\leq e\leq m$, we have
$q_m\leq q_{s(e)}$, so \Cref{thm:stable-variable-embedding}
gives an isometric embedding of $D_e$ into $L^{q_m}[0,1]$.
The space $X_m=\ell_{q_m}^{d_m}$ also embeds isometrically into
$L^{q_m}[0,1]$ by \eqref{eq:normalized-step-embedding}.
Partition $[0,1]$ into $m+2$ intervals of equal length.  By affine
changes of variable, multiplication by $(m+2)^{1/q_m}$, and extension
by zero, realize these embeddings on distinct intervals.  Their sum
is an isometric embedding of $Y_m$ into $L^{q_m}[0,1]$.

Write $a_m\in\Emb(X_m,Y_m)$ and $b_{e,m}\in\Emb(D_e,Y_m)$
for the rational coordinate inclusions.
Apply \Cref{lem:finite-pair-package} once, with base $Y_m$,
$r=q_m$, $r'=q_{m+1}$, $\theta=2^{-m-4}$, and one
pair requirement for each $(e,h)\in\{0,\ldots,m\}^2$,
using domain $D_e$, $\beta_{\dim D_e,h}$, and $\xi=2^{-h-3}$.
For the resulting $\Phi:Y_m\to\ell_{q_{m+1}}^N$, set
\[
 d_{m+1}=N,\quad X_{m+1}=\ell_{q_{m+1}}^N,\quad
 I_m=\Phi a_m,\quad C_{e,m}=\Phi b_{e,m}.
\]
Then $I_m\in\Emb^{\log}_{2^{-m-4}}(X_m,X_{m+1})$.
Applying the pair conclusion to
$a_m\gamma,a_m\eta$ and restricting each alternative map to $X_m$
gives the following two properties:

\begin{tcolorbox}[breakable,colback=white,colframe=black,
boxrule=0.5pt,arc=0pt,left=6pt,right=6pt,top=6pt,bottom=6pt]
\begin{description}[leftmargin=3.2em,style=nextline]
\item[$(P_m)$]
For $e,h\leq m$ and
$\gamma,\eta\in\Emb^{\log}_{\beta_{\dim D_e,h}}(D_e,X_m)$,
there is a rational $J\in\Emb^{\log}_{2^{-m-4}}(X_m,X_{m+1})$ such that
$\|J\gamma-I_m\eta\|<2^{-h-3}$.
\item[$(C_m)$]
For each $e\leq m$, there is a rational
$C_{e,m}\in\Emb^{\log}_{2^{-m-4}}(D_e,X_{m+1})$.
\end{description}
\end{tcolorbox}

Write $I_{m,t}=I_{t-1}\cdots I_m$ for $t>m$ and $I_{m,m}=\id$.
For $x\ne0$, successive terms of $\log\|I_{m,t}x\|_{X_t}$ differ
by less than $2^{-t-4}$.  Thus
\begin{equation}\label{eq:limit-norm}
 \|x\|_\infty=\lim_{t\to\infty}\|I_{m,t}x\|_{X_t}
\end{equation}
defines a norm on $X_m$, and
\begin{equation}\label{eq:stage-limit-distortion}
 e^{-2^{-m-3}}\|x\|_{X_m}\leq\|x\|_\infty
 \leq e^{2^{-m-3}}\|x\|_{X_m}.
\end{equation}
Let $\bar X_m$ denote $X_m$ with this limit norm.  Then
$I_m\in\Emb(\bar X_m,\bar X_{m+1})$.
Let $B$ be the
completion of the union $V=\bigcup_m \bar X_m$.  Denote the canonical
embeddings by $\iota_m\in\Emb(\bar X_m,B)$.

\begin{lemma}\label{prop:limit-computable}
$B$ has a computable presentation.
\end{lemma}

\begin{proof}
Enumerate the pairs $(m,v)$ with $v\in\Q_\F^{d_m}$.
Their images $\iota_mv$ are dense in $B$, and any rational linear
combination can be represented by a rational vector $w$ in a
single space $X_m$.  Choose a rational $R>\|w\|_{X_m}$ and put
$\Delta_t=\|I_{m,t}w\|_{X_t}$ for $t\geq m$.
These numbers are uniformly computable, and
$e^{-2^{-t-3}}\Delta_t\leq\|\iota_mw\|_B\leq e^{2^{-t-3}}\Delta_t$.
Since $2^{-m-3}\leq1/4$, we have $\Delta_t<2R$; also
$e^{2^{-t-3}}-1\leq2^{-t-2}$.  Consequently
\[
 \bigl|\|\iota_mw\|_B-\Delta_t\bigr|
 \leq(e^{2^{-t-3}}-1)\Delta_t
 \leq R2^{-t-1}\leq R2^{-t}.
\]
This gives an effective error bound for the norm of every rational
linear combination, as required.
\end{proof}

\begin{lemma}\label{lem:stage-geometry}
The following properties hold.
\begin{enumerate}[label=(\roman*)]
\item For every $m$, $\bar X_m \in \Age(L^p[0,1])$.
\item For every $E\in\Age(L^p[0,1])$ and every $\alpha>0$,
all sufficiently large $t$ admit a map
$G\in\Emb^{\log}_{\alpha}(E,X_t)$.
\end{enumerate}
\end{lemma}

\begin{proof}
\textup{(i)} Fix $m$.  For $t\geq m$, the isometricity of the bonding maps
for the limiting norms and \eqref{eq:stage-limit-distortion} give
\[
 e^{-2^{-t-3}}\|x\|_{\bar X_m}
 \leq\|I_{m,t}x\|_{X_t}
 \leq e^{2^{-t-3}}\|x\|_{\bar X_m}.
\]
Thus the image $I_{m,t}(X_m)$, with the norm inherited from $X_t$,
has Banach--Mazur distance at most $2^{-t-2}$ from $\bar X_m$.
Every $X_t=\ell_{q_t}^{d_t}$ belongs to $\Age(L^p[0,1])$ by
the normalized-indicator embedding \eqref{eq:normalized-step-embedding}
and \Cref{thm:stable-variable-embedding}.
Since $2^{-t-3}\to0$, the closedness of the fixed-dimensional age in
\Cref{lem:lp-age-closed} proves the assertion.

\textup{(ii)} Fix $E\in\Age(L^p[0,1])$ and $\alpha>0$.
Use \Cref{lem:D-dense} and \Cref{lem:bm-versus-ldis} to choose
$R\in\Emb^{\log}_{\alpha/2}(E,D_d)$.  For sufficiently large
$t$, we have $t-1\geq d$ and $2^{-t-3}<\alpha/2$; then
$(C_{t-1})$ gives the required map $G=C_{d,t-1}R$.
\end{proof}

\begin{proof}[Proof of \Cref{thm:main}]
By \Cref{prop:limit-computable}, it remains to identify $B$ with
$L^p([0,1];\F)$.  We do this by an approximate intertwining.

Fix a sequence $(z_i)_{i\in\N}$ dense in the unit ball of $L^p[0,1]$.
We construct strictly increasing indices $m_n$ and
$f_n\in\Emb(\bar X_{m_n},L^p[0,1])$ such that
\begin{equation}\label{eq:identification-coherence}
 \|f_{n+1}I_{m_n,m_{n+1}}-f_n\|_{\bar X_{m_n},L^p}
 <3\cdot2^{-n-3}.
\end{equation}
At the same time, writing
$E_n=\operatorname{span}\bigl(f_n(\bar X_{m_n}),z_0,\ldots,z_n\bigr)$,
we construct maps $b_n:E_n\to\bar X_{m_{n+1}}$ with $\|b_n\|<2$ and 
\begin{equation}\label{eq:identification-backward}
 \|f_{n+1}b_n-\id_{E_n}\|_{E_n,L^p}<2^{-n-3}.
\end{equation}
Here $E_n$ always carries its norm as a subspace of $L^p[0,1]$.
Start with $m_0=0$ and any
$f_0\in\Emb(\bar X_0,L^p[0,1])$, which exists by \Cref{lem:stage-geometry}(i).

Suppose $m_n,f_n$ have been chosen, and define $E_n$ as above.
Put $\beta=\beta_{d_{m_n},n}$.  Since
$\bar X_{m_n}\in\Age(L^p[0,1])$, choose a $d_{m_n}$-dimensional
$D_e$ and an isomorphism $S\in\Emb^{\log}_{\beta/4}(D_e,\bar X_{m_n})$
using \Cref{lem:D-dense} and \Cref{lem:bm-versus-ldis}.
Choose a positive rational 
$\delta<\min\{\beta/2,
              \omega_I(\dim E_n,2^{-n-3})\}$.
By \Cref{lem:stage-geometry}(ii), we may take $t>m_n$ sufficiently
large that $t\geq e,n$, $2^{-t-3}<\delta/2$, and there is
$g\in\Emb^{\log}_{\delta/2}(E_n,X_t)$.
We apply $(P_t)$ with domain $D_e$, accuracy index $h=n$, and maps
$gf_nS$ and $I_{m_n,t}S$.  Their eligibility follows from
\[
 \begin{aligned}
 gf_nS&\in
 \Emb^{\log}_{\delta/2+\beta/4}(D_e,X_t)
 \subseteq\Emb^{\log}_{\beta}(D_e,X_t),\\
 I_{m_n,t}S&\in
 \Emb^{\log}_{2^{-t-3}+\beta/4}(D_e,X_t)
 \subseteq\Emb^{\log}_{\beta}(D_e,X_t).
 \end{aligned}
\]
Thus $(P_t)$ supplies a rational
$J\in\Emb^{\log}_{2^{-t-4}}(X_t,X_{t+1})$ satisfying
\begin{equation}\label{eq:identification-stage-alignment}
 \|Jgf_nS-I_tI_{m_n,t}S\|_{D_e,X_{t+1}}<2^{-n-3}.
\end{equation}

Set $m_{n+1}=t+1$ and regard $b_n=Jg$ as a map into
$\bar X_{m_{n+1}}$.  Since $2\cdot2^{-t-4}=2^{-t-3}<\delta/2$, we have
\[
 b_n\in
 \Emb^{\log}_{\delta/2+2^{-t-3}}
       (E_n,\bar X_{m_{n+1}})
 \subseteq\Emb^{\log}_{\delta}(E_n,\bar X_{m_{n+1}}).
\]
In particular, $\|b_n\|<e^{\delta}<2$.  Composing
\eqref{eq:identification-stage-alignment} with $S^{-1}$ and
converting its target norm by \eqref{eq:stage-limit-distortion} gives
\begin{equation*}
 \begin{aligned}
 \|b_nf_n-I_{m_n,m_{n+1}}\|_{\bar X_{m_n},\bar X_{m_{n+1}}}
 &<e^{2^{-t-4}}\,2^{-n-3}\|S^{-1}\|\\
 &<e^{2^{-t-4}+\beta/4}\,2^{-n-3}<2^{-n-2},
 \end{aligned}
\end{equation*}
where $2^{-t-4}\leq1/8$ and $\beta<1$ justify the last bound.

Choose $j_n\in\Emb(\bar X_{m_{n+1}},L^p[0,1])$ by \Cref{lem:stage-geometry}(i).
The maps $j_nb_n$ and $\id_{E_n}$ belong to
$\Emb^{\log}_{\omega_I(\dim E_n,2^{-n-3})}(E_n,L^p[0,1])$.
Since $p\in I$ and $E_n\in\Age(L^p[0,1])$,
\Cref{thm:uniform-stability} supplies $W_n\in\Iso(L^p[0,1])$
such that
\[
 \|W_nj_nb_n-\id_{E_n}\|_{E_n,L^p}<2^{-n-3}.
\]
Put $f_{n+1}=W_nj_n\in\Emb(\bar X_{m_{n+1}},L^p[0,1])$.
This proves \eqref{eq:identification-backward}.  Furthermore,
\[
 \begin{aligned}
 \|f_{n+1}I_{m_n,m_{n+1}}-f_n\|
 &\leq
 \|f_{n+1}(I_{m_n,m_{n+1}}-b_nf_n)\|
 +\|(f_{n+1}b_n-\id_{E_n})f_n\|\\
 &<2^{-n-2}+2^{-n-3}=3\cdot2^{-n-3},
 \end{aligned}
\]
which establishes \eqref{eq:identification-coherence} and completes
the induction.

For $x\in\bar X_m$, define
\[
 F\iota_mx=\lim_{\substack{n\to\infty\\m_n\geq m}}
                  f_nI_{m,m_n}x.
\]
The limit exists by \eqref{eq:identification-coherence} and the
summability of $(2^{-n-3})$; the identifications along the bonding maps make $F$ a
well-defined linear map on the dense union $V$.  Each
$f_nI_{m,m_n}$ is isometric for the limiting norm, so $F$ extends
to $F\in\Emb(B,L^p[0,1])$.
Telescoping \eqref{eq:identification-coherence} also gives
\begin{equation}\label{eq:identification-tail}
 \|F\iota_{m_n}-f_n\|_{\bar X_{m_n},L^p}
 \leq\sum_{j\geq n}3\cdot2^{-j-3}=6\cdot2^{-n-3}.
\end{equation}

To prove surjectivity, fix $i$ and let $n\geq i$.  Then
$z_i\in E_n$ and $\|z_i\|\leq1$.  By
\eqref{eq:identification-backward} and \eqref{eq:identification-tail},
\[
 \begin{aligned}
 \|F\iota_{m_{n+1}}b_nz_i-z_i\|
 &\leq 6\cdot2^{-n-4}\|b_nz_i\|
       +\|f_{n+1}b_nz_i-z_i\|\\
 &<12\cdot2^{-n-4}+2^{-n-3}=7\cdot2^{-n-3}\to0.
 \end{aligned}
\]
The range $F(B)$ is closed because $F$ is an isometric embedding
and $B$ is complete.  It therefore contains every $z_i$, and hence the whole
unit ball by density.  As $F(B)$ is a linear subspace, this proves
$F(B)=L^p[0,1]$.
\end{proof}

\section{Exponent recovery in the atomic case and the classification}
\label{sec:exponent-classification}

We prove \Cref{thm:ellp-recovery} by extracting $p$ from symmetric
second differences of the norm.  The argument uses only one atomic
coordinate, so we formulate its geometric estimate for general
$L^p$-spaces for use in the general classification below.

\begin{lemma}\label{lem:ellp-log-recovery}
Let $r\in\R$.  Suppose $(\beta_n)_{n\ge1}$ is a uniformly computable
sequence of positive reals and there are constants $c,C>0$ and
$N\in\N$ such that
$c2^{-nr}\le \beta_n\le C2^{-nr}$ for all $n\ge N$.
Then $r$ is computable.
\end{lemma}

\begin{proof}
The reals $\alpha_n=-n^{-1}\log_2\beta_n$ are uniformly computable,
and
$|\alpha_n-r|\le D/n$ for $n\ge N$ and some fixed $D>0$.
\end{proof}

We next record the form of Clarkson's inequality \cite{Clarkson}
used below, together with an elementary proof.

\begin{lemma}[Clarkson's inequality]\label{lem:lp-clarkson}
Let $1\le p\le2$, and let $B$ be a real or complex $L^p$-space.
For all $u,v\in B$,
\begin{equation}\label{eq:lp-clarkson}
 \lVert u+v\rVert_B^p+\lVert u-v\rVert_B^p
 \le2\bigl(\lVert u\rVert_B^p+\lVert v\rVert_B^p\bigr).
\end{equation}
If $u\ne0$, then also
\begin{equation}\label{eq:lp-norm-second-difference}
 0\le\lVert u+v\rVert_B+\lVert u-v\rVert_B-2\lVert u\rVert_B
 \le\frac{2}{p}\lVert u\rVert_B^{1-p}\lVert v\rVert_B^p.
\end{equation}
\end{lemma}

\begin{proof}
Put $q=p/2\in(0,1]$.  For $a,b\ge0$, $(a+b)^q\le a^q+b^q$.
For real or complex scalars $\alpha,\beta$, concavity of
$t\mapsto t^q$ and this inequality give
\[
 \begin{aligned}
 |\alpha+\beta|^p+|\alpha-\beta|^p
 &\le2\left(\frac{|\alpha+\beta|^2+|\alpha-\beta|^2}{2}\right)^q\\
 &=2\bigl(|\alpha|^2+|\beta|^2\bigr)^q
 \le2\bigl(|\alpha|^p+|\beta|^p\bigr).
 \end{aligned}
\]
Applying this scalar inequality pointwise to representatives of
$u$ and $v$ and integrating proves \eqref{eq:lp-clarkson}.

Now suppose $u\ne0$, and put $S=\lVert u\rVert_B^p>0$ and
$G(s)=s^{1/p}$.  Concavity of $G$ gives the tangent inequality $G(t)-G(S)\le G'(S)(t-S)$ for $t\ge0$.
Apply it with $t=\lVert u+v\rVert_B^p$ and
$t=\lVert u-v\rVert_B^p$, and add.  By \eqref{eq:lp-clarkson},
\[
 \begin{aligned}
 \lVert u+v\rVert_B+\lVert u-v\rVert_B-2\lVert u\rVert_B
 &\le G'(S)\bigl(\lVert u+v\rVert_B^p
                          +\lVert u-v\rVert_B^p-2S\bigr)\\
 &\le2G'(S)\lVert v\rVert_B^p
 =\frac{2}{p}\lVert u\rVert_B^{1-p}\lVert v\rVert_B^p.
 \end{aligned}
\]
The lower bound in \eqref{eq:lp-norm-second-difference} follows
from the triangle inequality.
\end{proof}

\begin{lemma}\label{lem:lp-midpoint-defects}
Let $1<p<2$, let $B$ be an $L^p$-space over $\F$, and let $x,y\in B$
be $\F$-linearly independent.  Put $\varphi(\lambda)=\lVert x+\lambda y\rVert_B$ for $\lambda\in\F$,
and for real $h>0$, let
\[
 \Delta_h\varphi(\lambda)
 =\varphi(\lambda+h)+\varphi(\lambda-h)-2\varphi(\lambda).
\]
For $n\ge1$, define
\[
 K_n=\begin{cases}
 [-n,n],&\F=\R,\\
 \{\lambda\in\C:|\operatorname{Re}\lambda|,
                         |\operatorname{Im}\lambda|\le n\},&\F=\C,
 \end{cases}
\]
and
$M_n=\max_{\lambda\in K_n}\Delta_{2^{-n}}\varphi(\lambda)$.
Then $0<M_n\le C2^{-np}$ for some $C>0$.  If there is a linear
isometric identification $B\cong\F\oplus_pY$, with
$\lVert(a,w)\rVert^p=|a|^p+\lVert w\rVert_Y^p$, under which
$y=(y_0,y')$ with $y_0\ne0$, then also
$M_n\ge c2^{-np}$ for some $c>0$ and all sufficiently large $n$.
If $x,y$ are finite $\Q_{\F}$-linear combinations of the elements of
a computable presentation $(b_j)_{j\in\N}$ of $B$, then $(M_n)$ is
uniformly computable.
\end{lemma}

\begin{proof}
Fix $h>0$ and put $u=x+hy$ and $v=x-hy$.  These vectors are
linearly independent, since $x=(u+v)/2$ and $y=(u-v)/(2h)$ are
linearly independent.  Thus $a=\lVert u\rVert_B$ and
$b=\lVert v\rVert_B$ are positive, and $u/a$ and $v/b$ are distinct
unit vectors.  Since $B$ is strictly convex and $a/(a+b)\in(0,1)$,
\[
 \frac{2\lVert x\rVert_B}{a+b}
 =\left\lVert\frac{a}{a+b}\frac{u}{a}
              +\frac{b}{a+b}\frac{v}{b}\right\rVert_B<1.
\]
Consequently,
\[
 \Delta_h\varphi(0)
 =\lVert x+hy\rVert_B+\lVert x-hy\rVert_B-2\lVert x\rVert_B
 =a+b-2\lVert x\rVert_B>0.
\]
Since $0\in K_n$, this also gives $M_n>0$ for every $n\ge1$.
Moreover, $\varphi$ is $\lVert y\rVert_B$-Lipschitz, so
$\Delta_h\varphi$ is $4\lVert y\rVert_B$-Lipschitz, uniformly in $h$.
Under the additional assumption on $x,y$, the norm algorithm computes
$\varphi(\lambda)$ uniformly for $\lambda\in\Q_{\F}$ and computes
$\lVert y\rVert_B$.  The Lipschitz bound therefore makes $\varphi$
a computable function on $\F$.  Since the sets $K_n$ are uniformly
computably compact, $M_n$ is computable uniformly in $n$, by
\Cref{fact:effective-compactness-closure}(ii).

For the upper estimate, put
$\delta=\operatorname{dist}(x,\F y)>0$.
Apply \Cref{lem:lp-clarkson} with $\xi=x+\lambda y$ and
$\eta=hy$.  Since $\lVert\xi\rVert_B\ge\delta$, we obtain
\begin{align*}
 \Delta_h\varphi(\lambda)
 \le \frac{2}{p}\lVert\xi\rVert_B^{1-p}\lVert\eta\rVert_B^p
 \le \frac{2}{p}\delta^{1-p}\lVert y\rVert_B^p h^p.
\end{align*}

For the lower estimate, write $x=(x_0,x')$, choose
$\lambda_0=-x_0/y_0$, and put
$z=x+\lambda_0y=(0,z')$.  Linear independence implies $z\ne0$.
Put $S_0=\lVert z\rVert_B^p=\lVert z'\rVert_Y^p>0$ and
$d_\pm(h)=\lVert z\pm hy\rVert_B^p-S_0$.
Since $z\pm hy=(\pm hy_0,z'\pm hy')$, the direct-sum norm gives
\[
 d_\pm(h)=|y_0|^p h^p+\lVert z'\pm hy'\rVert_Y^p
                         -\lVert z'\rVert_Y^p.
\]
The function $t\mapsto\lVert z'+ty'\rVert_Y^p$ is convex on $\R$,
so its values at $h$ and $-h$ have sum at least twice its value at
$0$.  Consequently,
\[
 \begin{aligned}
 d_+(h)+d_-(h)
 =2|y_0|^p h^p+\lVert z'+hy'\rVert_Y^p
                  +\lVert z'-hy'\rVert_Y^p-2\lVert z'\rVert_Y^p
 \ge2|y_0|^p h^p.
 \end{aligned}
\]

We next bound each $d_\pm(h)$ separately.  Write
$R=\lVert z\rVert_B>0$ and restrict to
$0<h\le R/(2\lVert y\rVert_B)$.  The triangle inequality
gives $\bigl|\lVert z\pm hy\rVert_B-R\bigr|\le h\lVert y\rVert_B$, hence 
$\frac R2\le\lVert z\pm hy\rVert_B\le\frac{3R}{2}$.
The derivative of $t\mapsto t^p$ is bounded by
$p(3R/2)^{p-1}$ on $[R/2,3R/2]$.  Thus the mean value theorem yields
\[
 \begin{aligned}
 |d_\pm(h)|
 \le p(3R/2)^{p-1}
       \bigl|\lVert z\pm hy\rVert_B-R\bigr|
 \le p(3R/2)^{p-1}\lVert y\rVert_Bh=:C_1h.
 \end{aligned}
\]
In particular, $d_\pm(h)=O(h)$, with a constant independent of $h$
and of the choice of sign.

Put $G(s)=s^{1/p}$.  Both $S_0$ and $S_0+d_\pm(h)$ belong to
the compact interval $I=[(R/2)^p,(3R/2)^p]\subset(0,\infty)$.
Since $G''(s)=-\frac{p-1}{p^2}s^{1/p-2}$,
the constant $C_2=\max_{s\in I}|G''(s)|$ is finite.
Taylor's theorem, applied separately to the two increments, gives
\[
 \begin{gathered}
 G(S_0+d_\pm(h))
 =G(S_0)+G'(S_0)d_\pm(h)+r_\pm(h),\\
 |r_\pm(h)|\le\frac{C_2}{2}|d_\pm(h)|^2
             \le\frac{C_2C_1^2}{2}h^2.
 \end{gathered}
\]
Using $G(S_0+d_\pm(h))=\varphi(\lambda_0\pm h)$ and
$G(S_0)=\varphi(\lambda_0)$, we obtain
\[
 \begin{aligned}
 \Delta_h\varphi(\lambda_0)
 &=G'(S_0)\bigl(d_+(h)+d_-(h)\bigr)+r_+(h)+r_-(h)\\
 &\ge2G'(S_0)|y_0|^p h^p-C_2C_1^2h^2\\
 &=h^p\bigl(2G'(S_0)|y_0|^p-C_2C_1^2h^{2-p}\bigr).
 \end{aligned}
\]
Here $G'(S_0)|y_0|^p>0$, and $h^{2-p}\to0$ as $h\to0^+$ because
$p<2$.  Hence, for all sufficiently small $h>0$,
$C_2C_1^2h^{2-p}\le G'(S_0)|y_0|^p$, and therefore $\Delta_h\varphi(\lambda_0)\ge ch^p$ where $c=G'(S_0)|y_0|^p>0$.
Finally, $\lambda_0\in K_n$ and $2^{-n}$ is sufficiently small for
all sufficiently large $n$, so
$M_n\ge\Delta_{2^{-n}}\varphi(\lambda_0)\ge c2^{-np}$.
\end{proof}

A single atom suffices for the lower estimate in
\Cref{lem:lp-midpoint-defects}, giving the following recovery result.

\begin{theorem}
\label{prop:general-atomic-recovery}
Let $\F$ be either $\R$ or $\C$, let $1<p<2$, and let $B$ be a
separable $L^p$-space over $\F$ of dimension at least two.  If $B$ has
a nonzero atomic part and is computably presentable, then $p$ is
computable.
\end{theorem}

\begin{proof}
Isolate one normalized atomic coordinate to write
$B\cong\F\oplus_pY$, where $Y\ne\{0\}$.
Fix a computable presentation $(b_j)_{j\in\N}$ of $B$ and let $D$
be its $\Q_{\F}$-linear span, which is dense in $B$.
Since $\{0\}\oplus_pY$ is a proper closed subspace, choose
$y\in D\setminus(\{0\}\oplus_pY)$.  Since $\dim B\ge2$, the subspace
$\F y$ is also proper and closed, so choose $x\in D\setminus\F y$.
Thus $y_0\ne0$ in the chosen decomposition and $x,y$ are linearly
independent.
By \Cref{lem:lp-midpoint-defects}, these vectors yield a uniformly
computable positive sequence $(M_n)$ with $c2^{-np}\le M_n\le C2^{-np}$ for some $c,C>0$ and sufficiently large $n$.
\Cref{lem:ellp-log-recovery} now makes $p$ computable.
\end{proof}

In particular, taking $B=\ell^p(\F)$ proves \Cref{thm:ellp-recovery}.
We now combine the exponent-recovery results with the classical
classification of separable $L^p$-spaces.  Write $E\oplus_p F$ for
the sum with norm
$\lVert(x,y)\rVert_{E\oplus_p F}
 =\bigl(\lVert x\rVert_E^p+\lVert y\rVert_F^p\bigr)^{1/p}$.

\begin{theorem}[see \cite{CembranosMendoza}]
\label{thm:classical-separable-classification}
Let $1\le p<\infty$, $p\ne2$, and let $\F$ be either $\R$ or $\C$.
Every nonzero separable $L^p$-space over $\F$ is linearly isometric to
exactly one of the following:
\begin{enumerate}[label=\textup{(\arabic*)}]
 \item $\ell_n^p(\F)$ for some $n\ge1$;
 \item $\ell^p(\F)$;
 \item $L^p([0,1];\F)$;
 \item $\ell_n^p(\F)\oplus_p L^p([0,1];\F)$ for some $n\ge1$;
 \item $\ell^p(\F)\oplus_p L^p([0,1];\F)$.
\end{enumerate}
\end{theorem}

At $p=2$, these types collapse to the usual Hilbert-space
classification.

\begin{fact}
\label{prop:hilbert-classification}
Every nonzero separable $L^2$-space over $\F$ is linearly isometric to
exactly one of $\ell_n^2(\F)$ ($n\ge1$) or $\ell^2(\F)$.
\end{fact}

\begin{theorem}
\label{thm:complete-separable-classification}
Let $\F$ be either $\R$ or $\C$, let $1\le p<\infty$, and let $B$ be a
nonzero separable $L^p$-space over $\F$.  Then $B$ is computably
presentable if and only if exactly one of the following alternatives
holds:
\begin{enumerate}[label=\textup{(\roman*)}]
 \item $\dim B=1$;
 \item $\dim B\ge2$, $1<p<2$, $B$ is nonatomic, and $p$ is
       right-c.e.;
 \item $\dim B\ge2$, $1<p<2$, $B$ has a nonzero atomic part, and
       $p$ is computable;
 \item $\dim B\ge2$ and $p\in\{1,2\}$;
 \item $\dim B\ge2$, $2<p<\infty$, and $p$ is computable.
\end{enumerate}
\end{theorem}

\begin{proof}
In dimension one, $B\cong\F$ has a computable presentation
independently of $p$.  Assume $\dim B\ge2$.
Whenever $p$ is computable, the standard presentations of all
types in \Cref{thm:classical-separable-classification} are
computable, including the indicated $L^p$-sums;
\Cref{prop:hilbert-classification} supplies the case $p=2$.
If $1<p<2$ and $B$ is nonatomic, then
$B\cong L^p([0,1];\F)$.  In this case,
\Cref{thm:one-sided-exponent-recovery,thm:main} show that computable
presentability is equivalent to $p$ being right-c.e.
In every remaining case, computable
presentability forces $p$ to be computable, by
\Cref{prop:general-atomic-recovery} for $1<p<2$ and
\Cref{thm:mcnicholl-exponent-recovery}(i) for $p>2$.
\end{proof}

\section{An effective Fra\"{\i}ss\'e theorem for Banach spaces}
\label{sec:gurarij-presentations}

We extend the finite-dimensional construction method of
\Cref{sec:proof-main} to a fixed family of Banach spaces under
abstract hypotheses.  The construction here uses only linear
isometric embeddings, so there is no accumulated distortion to
control and no need to estimate convergence to a limiting norm.
This makes the construction technically simpler than the one in
\Cref{sec:proof-main}.  It is a Banach-space analogue of the
effective Fraïssé constructions for countable structures in Csima,
Harizanov, Miller and Montalbán \cite{CHMM}.  Its application to the Gurarij
Banach space gives a complementary computable presentability
result.

Throughout this section, all Banach spaces and linear maps are over
$\R$.  Let $\mathcal A=(A_n)_{n\in\N}$ be a countable family
of nonzero finite-dimensional real Banach spaces.

A chain of linear isometric embeddings is a sequence $U_0\xrightarrow{u_0}U_1\xrightarrow{u_1}\cdots$ such that $U_s\in\mathcal A$ and $u_s\in\Emb(U_s,U_{s+1})$.
Write $u_r^s=u_{s-1}\cdots u_r$ and $u_r^r=\id_{U_r}$.
When the spaces in $\mathcal A$ are given in fixed coordinates,
the chain is called \emph{computable} if $U_s=A_{n_s}$ for a
computable sequence $(n_s)$ and the matrix entries of $u_s$
are computable uniformly in $s$.
The chain is an \emph{$\mathcal A$-Fra\"{\i}ss\'e sequence} if:
\begin{enumerate}[label=(F\arabic*)]
\item\label{item:fraisse-cofinal}
Every $A_p$ embeds linearly isometrically into some $U_s$.
\item\label{item:fraisse-extension}
For every $f\in\Emb(U_r,A_p)$ and $\epsilon>0$, there are $s>r$
and $i\in\Emb(A_p,U_s)$ with
$\|if-u_r^s\|<\epsilon$.
\end{enumerate}

\begin{theorem}
\label{thm:effective-exact-fraisse}
Suppose that $\mathcal A=(A_n)_{n\in\N}$ satisfies:
\begin{enumerate}[label=(E\arabic*)]
\item\label{item:effective-objects}
Each $A_n$ is given in coordinates as $(\R^{d_n},\|\cdot\|_{A_n})$,
where $d_n$ is computable from $n$, and $\|x\|_{A_n}$ is computable
uniformly in $n$ and $x\in\Q^{d_n}$.
\item\label{item:effective-arrows}
There are computable sequences $(p_e)_{e\in\N}$ and $(q_e)_{e\in\N}$
and $T_e \in \Emb(A_{p_e},A_{q_e})$ whose
matrix entries are computable uniformly in $e$, such that
\[
 \Emb_{\mathrm c}(A_p,A_q)
 :=\{T_e:p_e=p,\ q_e=q\}
\]
is dense in $\Emb(A_p,A_q)$ in the operator norm for every $p,q$.
\item\label{item:abstract-jep}
Every two spaces in $\mathcal A$ embed linearly isometrically into
a common member of $\mathcal A$.
\item\label{item:abstract-nap}
For $f\in\Emb(A_r,A_p)$, $g\in\Emb(A_r,A_q)$, and $\eta>0$,
there are $A_t$ and $i\in\Emb(A_p,A_t)$, $j\in\Emb(A_q,A_t)$ such that
$\|if-jg\|<\eta$.
\end{enumerate}
Then $\mathcal A$ has a computable $\mathcal A$-Fra\"{\i}ss\'e
sequence $U_0\xrightarrow{u_0}U_1\xrightarrow{u_1}\cdots$ whose completed direct limit is computably presentable.  Any two $\mathcal A$-Fra\"{\i}ss\'e sequences have linearly
isometric completed direct limits, hence their common limit
$\operatorname{Flim}(\mathcal A)$ has a computable Banach
presentation.
\end{theorem}

\begin{proof}
We shall construct indices $n_r$ and
$u_r\in\Emb(A_{n_r},A_{n_{r+1}})$, and write $U_r=A_{n_r}$.  For
\ref{item:fraisse-cofinal}, impose
\[
 \mathsf C_p:\quad
 \Emb(A_p,U_t)\ne\varnothing\quad\text{for some }t.
\]
By the density in \ref{item:effective-arrows},
\ref{item:fraisse-extension} reduces to meeting, whenever
$p_e=n_r$ and $k\ge1$, the requirements
\begin{equation}
\label{eq:abstract-extension-requirement}
 \mathsf E_{r,e,k}:\quad
 \|iT_e-u_r^t\|<2^{-k}
 \quad\text{for some }t>r,\ i\in\Emb(A_{q_e},U_t).
\end{equation}
Fix a computable listing without repetitions of all these formal
requirements, giving priority to earlier entries.  A requirement
is satisfied when the construction supplies a witness.

\smallskip
\noindent\emph{Construction.}
Set $U_0=A_0$ and $n_0=0$.  At stage $s+1$, act on the
highest-priority unsatisfied requirement that is either a
$\mathsf C_p$, or an $\mathsf E_{r,e,k}$ with
$r\le s$ and $p_e=n_r$. 

If it is $\mathsf C_p$, search for $a,b$ with $p_a=n_s$, $p_b=p$, $q_a=q_b$
and set $n_{s+1}=q_a$, $U_{s+1}=A_{q_a}$, and $u_s=T_a$.
Then $T_b\in\Emb(A_p,U_{s+1})$ is the required witness. 

If it is $\mathsf E_{r,e,k}$, search for $a,b$ with $p_a=q_e$, $p_b=n_s$, $q_a=q_b$ and $\|T_aT_e-T_bu_r^s\|<2^{-k}$,
and set $n_{s+1}=q_a$ and $u_s=T_b$.
Since $u_r^{s+1}=T_bu_r^s$, the map $T_a$ witnesses
$\mathsf E_{r,e,k}$. 

\smallskip
\noindent\emph{Verification.}
The index conditions are decidable, and the strict norm test is
semidecidable by \Cref{lem:effective-operator-extrema}.
The first search terminates by
\ref{item:abstract-jep} and \ref{item:effective-arrows}.
For the second, \ref{item:abstract-nap} gives an amalgam with
error less than $2^{-k-1}$; sufficiently close approximations to
its two embeddings from the dense families in
\ref{item:effective-arrows} satisfy the displayed test.
Thus all stages terminate and the bonding matrices are uniformly
computable.  Hence
every requirement is eventually satisfied.
The requirements $\mathsf C_p$ give \ref{item:fraisse-cofinal}.
For \ref{item:fraisse-extension}, given
$f\in\Emb(U_r,A_p)$ and $\epsilon>0$, choose $e,k$ with
$p_e=n_r$, $q_e=p$, $\|f-T_e\|<\epsilon/2$, and
$2^{-k}<\epsilon/2$.  A witness $i\in\Emb(A_p,U_t)$ to
$\mathsf E_{r,e,k}$ then satisfies
\[
 \|if-u_r^t\|
 \le\|f-T_e\|+\|iT_e-u_r^t\|<\epsilon.
\]

The norm computation in the proof of
\Cref{prop:limit-computable} gives a computable presentation of
the completed limit $B$.  Indeed, enumerate the images in $B$ of all
rational coordinate vectors in the spaces $U_s$.  They are dense, and,
writing $\iota_s\in\Emb(U_s,B)$ for the canonical embeddings, each
finite rational linear combination satisfies
\begin{equation}
\label{eq:abstract-limit-norm}
 \left\|\sum_{j<m}\alpha_j\iota_{s_j}a_j\right\|_B
 =\left\|\sum_{j<m}\alpha_j u_{s_j}^t a_j\right\|_{U_t},
 \qquad t\ge\max_{j<m}s_j.
\end{equation}
The right-hand vector has uniformly computable coordinates, so
its norm is uniformly computable by \ref{item:effective-objects}
and \eqref{eq:effective-norm-comparison}. 

Finally, let $(U_s,u_s)$ and $(V_s,v_s)$ be any two
$\mathcal A$-Fra\"{\i}ss\'e sequences.  We construct strictly
increasing indices $r_n,s_n$ and
\[
 f_n\in\Emb(U_{r_n},V_{s_n}),
 \qquad g_n\in\Emb(V_{s_n},U_{r_{n+1}})
\]
such that
\begin{equation}
\label{eq:abstract-approximate-intertwining}
 \begin{aligned}
 \|g_nf_n-u_{r_n}^{r_{n+1}}\|<2^{-n-2}, \quad
 \|f_{n+1}g_n-v_{s_n}^{s_{n+1}}\|<2^{-n-2}.
 \end{aligned}
\end{equation}
Set $r_0=0$.  Since $U_0\in\mathcal A$, property
\ref{item:fraisse-cofinal} for the sequence $(V_s,v_s)$ gives
an index $s_0$ and $f_0\in\Emb(U_0,V_{s_0})$.

Suppose that $r_n,s_n$ and $f_n\in\Emb(U_{r_n},V_{s_n})$ have
been chosen.  Apply \ref{item:fraisse-extension} to the sequence
$(U_s,u_s)$, taking $U_{r_n}$ as the source space, $V_{s_n}$
as the member of $\mathcal A$ serving as the target, $f_n$ as
the given embedding, and $2^{-n-2}$ as the tolerance.
This gives an index $r_{n+1}>r_n$ and
$g_n\in\Emb(V_{s_n},U_{r_{n+1}})$ satisfying the first
inequality in \eqref{eq:abstract-approximate-intertwining}.

Now apply \ref{item:fraisse-extension} to the sequence $(V_s,v_s)$,
with source $V_{s_n}$, target $U_{r_{n+1}}\in\mathcal A$,
given embedding $g_n$, and the same tolerance $2^{-n-2}$.
This gives an index $s_{n+1}>s_n$ and
$f_{n+1}\in\Emb(U_{r_{n+1}},V_{s_{n+1}})$ satisfying the
second inequality in \eqref{eq:abstract-approximate-intertwining}.
This completes the induction. 

Let $B_U$ and $B_V$ be the respective completed direct limits.
Identify $U_s$ and $V_s$ with their images under the canonical
embeddings, so that the bonding maps $(u_s)$ and $(v_s)$ become inclusions.
Since $(r_n)$ and $(s_n)$ are strictly increasing, the unions
$\bigcup_n U_{r_n}$ and $\bigcup_n V_{s_n}$ are dense in
$B_U$ and $B_V$, respectively.
For $x\in U_{r_n}$, insert $f_{n+1}g_nf_nx$ and use
\eqref{eq:abstract-approximate-intertwining} to obtain
\[
 \begin{aligned}
 \|f_{n+1}x-f_nx\|_{B_V}
 &\le\|f_{n+1}(x-g_nf_nx)\|_{B_V}
       +\|f_{n+1}g_nf_nx-f_nx\|_{B_V}\\
 &\le 2^{-n-2}\|x\|_{B_U}
       +2^{-n-2}\|f_nx\|_{B_V}
 =2^{-n-1}\|x\|_{B_U}.
 \end{aligned}
\]
For a fixed $x$ in $\bigcup_n U_{r_n}$, this estimate applies
for all sufficiently large $n$.  Its summable right-hand side
shows that $(f_nx)$ is Cauchy in $B_V$.  Define
$Fx=\lim_n f_nx$, starting with any $n$ for which $x\in U_{r_n}$.
The definition is independent of this starting index.  Linearity
follows by taking limits on a common space containing the vectors
involved, and
$\|Fx\|_{B_V}=\lim_n\|f_nx\|_{B_V}=\|x\|_{B_U}$.
Thus $F$ extends uniquely to $F\in\Emb(B_U,B_V)$.  Summing the difference estimate also gives
\[
 \|Fx-f_nx\|_{B_V}
 \le\sum_{m\ge n}2^{-m-1}\|x\|_{B_U}
 =2^{-n}\|x\|_{B_U}
 \qquad(x\in U_{r_n}).
\]

To prove surjectivity, fix $y\in V_{s_j}$ and let $n\ge j$.
Then $y\in V_{s_n}$, $g_ny\in U_{r_{n+1}}$, and
$\|g_ny\|_{B_U}=\|y\|_{B_V}$.  Applying the preceding tail
estimate with index $n+1$ and then the second inequality in
\eqref{eq:abstract-approximate-intertwining}, we get
\[
 \begin{aligned}
 \|F(g_ny)-y\|_{B_V}
 &\le\|F(g_ny)-f_{n+1}g_ny\|_{B_V}
       +\|f_{n+1}g_ny-y\|_{B_V}\\
 &\le 2^{-n-1}\|g_ny\|_{B_U}
       +2^{-n-2}\|y\|_{B_V}
 =3\cdot2^{-n-2}\|y\|_{B_V}\to 0.
 \end{aligned}
\]
Consequently every vector in the dense union $\bigcup_jV_{s_j}$
lies in the closure of $F(B_U)$.  This range is closed because
$F$ is isometric and $B_U$ is complete.  Hence $F(B_U)=B_V$,
which proves uniqueness of $\operatorname{Flim}(\mathcal A)$.
\end{proof}

We now apply \Cref{thm:effective-exact-fraisse} to the Gurarij
Banach space, following Ben Yaacov and Henson \cite{BenYaacovHenson}
for the definition and basic properties. 

\begin{definition}
\label{def:gurarij-space}
A nonzero real Banach space $G$ is a
\emph{Gurarij space} if, for every $\epsilon>0$, finite-dimensional
real Banach spaces $E\subseteq H$ with $H\ne\{0\}$, and
$u\in\Emb(E,G)$, there is 
$v \in \Emb^{\log}_\epsilon(H,G)$ such that 
$v\restr E=u$.
\end{definition}

A separable Gurarij space exists and is unique
up to surjective linear isometry \cite{LuskyGurarij}.
It is universal for separable real Banach spaces: every such space
embeds linearly isometrically into it
\cite[Corollary~2.7]{BenYaacovHenson}.
It is also approximately ultrahomogeneous
\cite[Corollary~2.8]{BenYaacovHenson}.
From now on, $\mathbb G$ denotes the unique separable real Gurarij
space.

\begin{question}[{\cite[Question A.19]{FGM}}]
\label{question:gurarij-presentations}
Does $\mathbb G$ have a computable
presentation?
\end{question}

For every integer $n\ge1$, let $\ell_\infty^n(\R)
 :=\bigl(\R^n,\|\cdot\|_\infty\bigr)$, where $\|(x_1,\ldots,x_n)\|_\infty=\max_{1\le j\le n}|x_j|$.
These finite-dimensional spaces with the maximum norm will be the
building blocks of our construction. A \emph{linear contraction}
is a linear map $T$ satisfying $\|Tx\|\le\|x\|$ for every $x$.

\Needspace{7\baselineskip}
\begin{lemma}
\label{lem:supnorm-amalgamation}
Let $E$ be a finite-dimensional real Banach space, let $m,n\ge1$,
and let $0\le\delta<1$.  Suppose that
$a:E\to\ell_\infty^m(\R)$ and $b:E\to\ell_\infty^n(\R)$
are linear contractions such that $(1-\delta)\|x\|_E
 \le \min\{\|ax\|_\infty,\|bx\|_\infty\}$ for all $x \in E$.
Then there are $i \in \Emb(\ell_\infty^m(\R),\ell_\infty^{m+n}(\R))$ and $j \in \Emb(\ell_\infty^n(\R),\ell_\infty^{m+n}(\R))$ such that $\|ia-jb\|\le\delta$.
\end{lemma}

\begin{proof}
The maps $a,b$ are injective, and
\[
 ax\mapsto(1-\delta)bx,
 \qquad
 bx\mapsto(1-\delta)ax
\]
are contractions on $a(E)$ and $b(E)$, respectively.
By applying the Hahn--Banach theorem coordinatewise, these maps
extend to contractions
\[
 R:\ell_\infty^m(\R)\to\ell_\infty^n(\R),
 \qquad
 S:\ell_\infty^n(\R)\to\ell_\infty^m(\R).
\]
Define $i(x)=(x,Rx)$ and $j(y)=(Sy,y)$, using the maximum norm
on the concatenated $m+n$ coordinates.  Both maps are isometric since $R$ and $S$ are contractions.
Since $Ra=(1-\delta)b$ and $Sb=(1-\delta)a$, we have $iax-jbx=(\delta ax,-\delta bx)$ and $\|iax-jbx\|_\infty\le\delta\|x\|_E$. 
This proves the lemma.
\end{proof}

\begin{lemma}[{\cite[Proposition~2.12(1),(5) and Section~6.1]{LupiniFraisse}}\footnote{This lemma is the equivalence of conditions (1) and (5) in
\cite[Proposition~2.12, p.~112]{LupiniFraisse}, specialized to real
Banach spaces as in \cite[Section~6.1, p.~130]{LupiniFraisse}, with
$\mathcal I=\{\ell_\infty^n(\R):n\ge1\}$.
Since the underlying structures are the closed unit balls, the
metric in \cite[Definition~2.3, p.~105]{LupiniFraisse} is
$d(g\varphi,f)=\sup_{\|x\|_E\le1}\|g\varphi x-fx\|_B
=\|g\varphi-f\|$.}]
\label{lem:gurarij-supnorm-characterization}
A nonzero separable real Banach space $B$ is linearly isometric to
$\mathbb G$ if and only if it has the following property:
for every finite-dimensional real Banach space $E$, integer
$n\ge1$, $f\in\Emb(E,B)$,
$\varphi\in\Emb(E,\ell_\infty^n(\R))$,
and $\epsilon>0$, there exists
$g\in\Emb(\ell_\infty^n(\R),B)$ such that
$\|g\varphi-f\|<\epsilon$.
\end{lemma}

\begin{theorem}
\label{thm:gurarij-presentations}
The real Gurarij Banach space $\mathbb G$ has a computable
presentation.
\end{theorem}

\begin{proof}
Apply
\Cref{thm:effective-exact-fraisse} to the explicitly indexed family $\mathcal A=(A_n)_{n\in\N}$ where 
$A_n=\ell_\infty^{n+1}(\R)$.
We use the standard coordinate basis of each $A_n$.  Its norm is
uniformly computable because it is the maximum of finitely many
absolute values.  This verifies \ref{item:effective-objects}.

We next construct the dense computable families of isometric
embeddings required by \ref{item:effective-arrows}.
Let $d,e\ge1$, and represent a linear map
$T:\ell_\infty^d(\R)\to\ell_\infty^e(\R)$ by an $e\times d$
matrix $(t_{ij})$.  Its operator norm is
$\|T\|=\max_{1\le i\le e}\sum_{j=1}^d|t_{ij}|$.
Consequently $T\in\Emb(\ell_\infty^d(\R),\ell_\infty^e(\R))$
if and only if both of the following conditions hold:
\begin{enumerate}[label=(\alph*)]
\item every row satisfies $\sum_{j=1}^d|t_{ij}|\le1$;
\item for every column $j$, some row has an entry equal to $1$ or $-1$
in column $j$ and has zero entries in all other columns.
\end{enumerate}
Indeed, if $T$ is isometric, then (a) follows from $\|T\|=1$.
The image of the $j$th coordinate unit vector has norm one,
so some entry in column $j$ equals $1$ or $-1$; (a) forces all
other entries in that row to vanish.  Conversely, (a) gives
$\|Tx\|_\infty\le\|x\|_\infty$, and the rows in (b) give
the reverse inequality for every $x$.

Conditions (a) and (b) are decidable for matrices with rational
entries.  Enumerating all rational matrices satisfying these
conditions, for all $d,e$, produces isometric embeddings whose
matrix entries are uniformly computable.

These embeddings are dense: keep one signed coordinate row as in (b)
for each column fixed, and approximate each remaining row by rational
points of the unit ball of $\ell_1^d$.  The resulting matrices still
satisfy (a) and (b) and converge in operator norm.  Thus their uniform
enumeration verifies \ref{item:effective-arrows}.

For \ref{item:abstract-jep}, two spaces
$\ell_\infty^d(\R)$ and $\ell_\infty^e(\R)$ embed into
$\ell_\infty^{d+e}(\R)$ by $x\mapsto(x,0)$ and
$y\mapsto(0,y)$.  For \ref{item:abstract-nap}, apply
\Cref{lem:supnorm-amalgamation}
with $\delta=0$. 
All four hypotheses of \Cref{thm:effective-exact-fraisse} hold.
It follows that there is a computable $\mathcal A$-Fra\"{\i}ss\'e
sequence $(U_r,u_r)$ whose completed limit $B$ has a computable
presentation.

It remains to verify that $B$ is $\mathbb G$.
Identify each $U_r$ with its canonical image in $B$; these are
increasing subspaces with dense union.  We check the property in
\Cref{lem:gurarij-supnorm-characterization}.
Fix a finite-dimensional $E$, $f\in\Emb(E,B)$,
$\varphi\in\Emb(E,\ell_\infty^n(\R))$, and
$\epsilon>0$.  Choose
$0<\eta<\min\{1/2,\epsilon/5\}$.
Approximating the images under $f$ of a basis of $E$ in the dense
union gives some $r$ and a linear map
$f_0:E\to U_r$ such that $\|f_0-f\|<\eta$.
Set $a=f_0/(1+\eta)$.  Then $a$ is a contraction,
$\|a-f\|<2\eta$, and $\|ax\|\ge(1-2\eta)\|x\|_E$ for all $x \in E$.
Write $U_r=\ell_\infty^m(\R)$.
By \Cref{lem:supnorm-amalgamation},
applied to $a,\varphi$ with $\delta=2\eta$, there are $i\in\Emb(U_r,\ell_\infty^{m+n}(\R))$ and $j\in\Emb(\ell_\infty^n(\R),\ell_\infty^{m+n}(\R))$
satisfying $\|ia-j\varphi\|\le2\eta$.
Property \ref{item:fraisse-extension} of the sequence, applied to
$i\in\Emb(U_r,\ell_\infty^{m+n}(\R))$, gives $s>r$ and
$h\in\Emb(\ell_\infty^{m+n}(\R),U_s)$ such that $\|hi-u_r^s\|<\eta$.
Regard $g=hj$ as an element of $\Emb(\ell_\infty^n(\R),B)$.
Since $u_r^s$ is the inclusion under our identifications and
$\|a\|\le1$, we obtain
\[
 \begin{aligned}
 \|g\varphi-f\|
 \le\|h(j\varphi-ia)\|
      +\|(hi-u_r^s)a\|+\|a-f\|
 <2\eta+\eta+2\eta
 =5\eta<\epsilon.
 \end{aligned}
\]
Thus $B$ satisfies \Cref{lem:gurarij-supnorm-characterization}
and is linearly isometric to $\mathbb G$.
Its computable presentation gives the desired presentation of
$\mathbb G$.
\end{proof}

\section*{Acknowledgments}The authors were supported by Focused Research Grants of Natural Science Foundation of Jiangsu Province No. BK20243060. 
The authors used OpenAI’s ChatGPT 5.6 to assist with developing the proofs. More specifically, it first suggested a possible connection between McNicholl’s conjecture and the work of Ferenczi, López-Abad, Mbombo, and Todor\v{c}evi\'c \cite{FLAMT}. When subsequently asked whether the relevant arguments could be effectivized, it provided the initial feasibility analysis that led to the approach developed in this paper. All
proofs in this paper have been independently verified by the authors without the use
of AI tools. The authors take full responsibility for the final content of the paper.

\end{document}